\documentclass[a4paper, 12pt]{amsart}

\usepackage{graphicx} 
\usepackage{amssymb, amsmath, amsthm}
\usepackage{amscd}
\usepackage{color}

\theoremstyle{plain}
\newtheorem{theorem}{Theorem}[section]
\newtheorem{lemma}[theorem]{Lemma}
\newtheorem{corollary}[theorem]{Corollary}
\newtheorem{definition}[theorem]{Definition}
\newtheorem{proposition}[theorem]{Proposition}

\newtheorem{example}[theorem]{Example}
\newtheorem{remark}[theorem]{Remark}

\makeatletter

 \@addtoreset{equation}{section}
\makeatother

\newcommand{\Lip}[0]{\mathrm{Lip}}
\newcommand{\Lipb}[0]{\mathrm{Lip}_{\mathrm b}}
\newcommand{\pa}[0]{\partial}
\newcommand{\cpt}[0]{\mathrm{c}}
\newcommand{\jump}[1]{\ensuremath{[\![#1]\!]} }
\newcommand{\spt}[0]{\mathrm{spt}}
\newcommand{\D}[0]{\mathcal D}
\newcommand{\M}[0]{\mathbf M}
\newcommand{\N}[0]{\mathbf N}
\newcommand{\I}[0]{\mathbf{I}}
\newcommand{\cab}[0]{\mathsf{C}\mathrm{(}\mathsf{Ab}\mathrm{)}}
\newcommand{\A}[0]{\mathfrak A}
\newcommand{\U}[0]{\mathcal U}
\newcommand{\Ab}[0]{\mathsf{Ab}}
\newcommand{\IC}[0]{\mathrm{IC}}
\renewcommand{\H}[0]{\mathfrak H}
\newcommand{\C}[0]{\mathfrak C}

\newcommand{\met}[0]{\mathsf{Met}}

\newcommand{\sing}[0]{\mathrm{sing}}

\title
[Cosheaves]
{The coincidence of the homologies of integral currents and of integral singular chains, via cosheaves} 
\author{Ayato Mitsuishi}
\date{\today}

\begin{document}
\maketitle

\begin{abstract}
We consider the notion of metric spaces being locally Lipschitz contractible introduced by Yamaguchi, and a category of metric spaces satisfying this condition. 
Many objects in metric geometry including CAT-spaces and Alexandrov spaces, belong to this category. 
We consider the homology of integral currents with compact support in a metric space, introduced by Ambrosio and Kirchheim, 
and prove that it and the usual integral singular homology are isomorphic on the category. 
The proof of it is based on the theory of cosheaves. 
A method to compare the homologies associated to cosheaves is also proved in this paper.
\end{abstract}

\section{Introduction}\label{sec:introduction}

There are homology theories of metric spaces depending on metric structures. 
For instance, we denote by $S_k^\Lip(X)$ the free abelian group based on the set of all Lipschitz maps from a $k$-simplex $\triangle^k$ with a standard Euclidean metric to a metric space $X$.
Then, it is a subgroup of the group of usual integral singular $k$-chains $S_k(X)$, and further, $S_\bullet^\Lip(X) = \bigoplus_{k=0}^\infty S_k^\Lip(X)$ becomes a subcomplex of $S_\bullet(X) = \bigoplus_{k=0}^\infty S_k(X)$.
By the definition, the homology $H_\ast^\Lip(X)$ of $S_\bullet^\Lip(X)$ depends on the metric structure of $X$. 
This group $H_\ast^\Lip(X)$ is called the {\it singular Lipschitz homology} of $X$. 
Here and hereafter, the symbol $\bullet$ means degrees of a chain complex 
and $\ast$ denotes a fixed degree (of a homology). 
Yamaguchi introduced the notion of metric spaces being locally Lipschitz contractible (abbreviated to LLC) and proved 
\begin{theorem}[\cite{Y}] \label{thm:Y}
For every LLC metric space $X$, the inclusion $S_\bullet^\Lip(X) \hookrightarrow S_\bullet(X)$ induces an isomorphsm $H_\ast^\Lip(X) \to H_\ast^\sing(X)$. 
\end{theorem}
Here, $H_\ast^\sing$ denotes the usual integral singular homology, that is the homology of $S_\bullet$. 
We can regard $S_\bullet^\Lip$ as a functor from the category of metric spaces and locally Lipschitz maps to the category of chain complices and chain maps, and the correspondence $S_\bullet^\Lip \to S_\bullet$ as a natural transformation between the functors. 

Ambrosio and Kirchheim introduced currents in general metric spaces (\cite{AK}). 
Currents are generalizations of operations of integrating smooth forms on submanifolds. 
We will recall the precise definition and fundamental properties of metric currents in Section \ref{sec:current}. 
A restricted class $\I_\bullet^\cpt(X)$ consisting of all integral currents with compact support in $X$ becomes a chain complex, due to \cite{AK}. 
Its homology is denoted by $H_\ast^\IC(X)$ in this paper. 
On the other hands, Riedweg and Sch\"appi introduced the notion of metric spaces admitting locally strong Lipschitz contractions (see Section \ref{sec:LLC}). 
They defined a natural transformation $[\,\cdot\,]$ from $S_\bullet^\Lip$ to $\I_\bullet^\cpt$ and claimed 
\begin{theorem}[\cite{RS}] \label{thm:RS}
On the category of metric spaces admitting locally strong Lipschitz contractions and locally Lipschitz maps, the natural transformations $S_\bullet \hookleftarrow S_\bullet^\Lip \xrightarrow{[\,\cdot\,]} \I_\bullet^\cpt$ induce isomorphisms between their homologies. 
\end{theorem}

The assumptions in Theorems \ref{thm:Y} and \ref{thm:RS} are related. 
Indeed, we will prove that a metric space admitting locally strong Lipschitz contractions is LLC (Lemma \ref{lem:LSLC to LLC}). 
A main result of the present paper is 
\begin{theorem} \label{thm:main thm}
On the category of LLC metric spaces and locally Lipschitz maps, the natural transformations $H_\ast^\sing \leftarrow H_\ast^\Lip \xrightarrow{[\,\cdot\,]_\ast} H_\ast^\IC$ are isomorphisms. 
\end{theorem}


As a direct corollary to Theorem \ref{thm:main thm}, we have
\begin{corollary} \label{cor:main cor}
The functors $H_\ast^\Lip$ and $H_\ast^\IC$ can be extended to functors on the category of all locally Lipschitz contractible metric spaces and all continuous maps, such that they are naturally isomorphic to $H_\ast^\sing$. 
In particular, $H_\ast^\Lip$ and $H_\ast^\IC$ are homotopy invariants for LLC metric spaces. 
\end{corollary}

The proofs of Theorems \ref{thm:Y} and \ref{thm:RS} were done directly. 
We will give a versatile proof of Theorem \ref{thm:main thm} using the theory of cosheaves.
Here, cosheaves are categorically dual notion of sheaves, which were introduced by Bredon \cite{Br}. 
Indeed, in the course of the proof of Theorem \ref{thm:main thm}, we give a technique to compare homologies associated to cosheaves (Theorem \ref{thm:cosheaf}). 
Notice that Mongodi \cite{Mo} considered that a chain complex consisting of metric currents in a metric space $X$, and proved that its homology coincides with the usual singular homology if $X$ is locally Lipschitz contractible {\it CW-complex}. 
His proof was done by verifying that the homology of currents is actually a homology theory, that is, it satisfies the axioms of Eilenberg and Steenrod. 
Further, he used the uniqueness of homology theory to prove the result. 
We should remark that there exists an LLC metric space which does not have the homotopy type of CW-complices (Section \ref{sec:remarks}). 
So, our result can not be proved via the uniqueness of homology.

A way to compare the homologies using cosheaves was also discussed by De~Pauw \cite{DP}. 
There, he consider the chain complices of usual currents in a subset of Euclidean spaces. 
Our proof of Theorem \ref{thm:main thm} is similar to an argument in that paper.
However, our formulation as in Theorem \ref{thm:cosheaf} did not appear. 
We consider that such a formulation is important and is very useful. 

\subsection{Organization}
The present paper mainly consists of three parts dealing with: Lipschitz contractions, metric currents, and cosheaves. 
The first two parts are subjects in geometry (and analysis) and the third part is purely an algebraic-topological subject.
Our main theorem in the paper is Therorem \ref{thm:main thm}. 
However, we consider that the proof of it is very important. 
It is based on Theorem \ref{thm:cosheaf}. 
If the reader purely is interested in algebraic topology, the author recommend to firstly read Section \ref{sec:cosheaf}.

In Section \ref{sec:LLC}, we review and define the notion of local Lipschitz contractibility and its variants. 
We prove that the local Lipschitz contractibility is weaker than other conditions.
In Section \ref{sec:current}, we review the notion of metric currents and its fundamental theory introduced and investigated by Ambrosio and Kirchheim (\cite{AK}). 
Further, we give proofs of several remarkable properties which are needed to prove our main results. 
In Section \ref{sec:cosheaf}, we recall the notion of cosheaves and its fundamental properties. 
We prove an important Theorem \ref{thm:cosheaf} which is a general method to compare homologies associated to coshaves. 
We generalize the local Lipschitz contractibility in terms of local triviality of homology theories (Lemma \ref{lem:LLC to LT}). 
At the end of the section, we prove Theorem \ref{thm:main thm} using Theorem \ref{thm:cosheaf} and Lemma \ref{lem:LLC to LT}. 
Finally, in Section \ref{sec:remarks}, we give several remarks about our results. 
In particular, we provide an example of an locally Lipschitz contractible metric space which does not admit the homotopy types of CW-complices.


\section{Variants of local Lipschitz contractibility} \label{sec:LLC}
Let us fix terminologies. 
In this section, $X$ and $Y$ always denote metric spaces. 
For $L \ge 0$, a map $f : X \to Y$ is said to be $L$-Lipschitz if it satisfies 
\[
d(f(x),f(y)) \le L d(x,y)
\]
for all $x,y \in X$. 
We say that $f$ is Lipschitz if $f$ is $L$-Lipschitz for some $L \ge 0$. 
The Lipschitz constant of $f$ is the minimum of all $L$ such that $f$ is $L$-Lipschitz, and is denoted by $\Lip(f)$.

A map $f : X \to Y$ is said to be locally Lipschitz if for any $x \in X$, there exists an open set $U$ in $X$ containing $x$ such that the restriction $f |_U$ is Lipschitz.

A map $f : X \to Y$ is called a {\it bi-Lipschitz embedding} if it satisfies 
\[
L^{-1} d(x,x') \le d(f(x), f(x')) \le L d(x,x')
\]
for all $x,x' \in X$ where $L \ge 1$ is some number.
If a bi-Lipschitz map is bijective, then it is called a {\it bi-Lipschitz homeomorphism}.
A {\it locally bi-Lipschitz homeomorphism} is a homeomorphism such that it and its inverse are locally Lipschitz.


A homotopy $h : X \times [0,1] \to Y$ is called a {\it Lipschitz homotopy} if it is a Lipschitz map, i.e., there exists a constant $C \ge 0$ such that 
\begin{equation*} 
d(h(x,t), h(x',t')) \le C (d(x,x') + |t-t'|)
\end{equation*}
for every $x,x' \in X$ and $t,t' \in [0,1]$.
For a homotopy $h$, we write $h_t = h(\cdot,t)$ for each $t \in [0,1]$.
We will also consider {\it locally Lipschitz homotopies}, which are homotopies being locally Lipschitz. 

Let $U$ and $V$ be subsets of $X$ with $U \subset V$. 
A Lipschitz homotopy $h : U \times [0,1] \to V$ is called a {\it Lipschitz contraction} if there exists a point $x_0 \in V$ such that $h_0$ is the inclusion $U \to V$ and $h_1 \equiv x_0$ is a constant map.
In this case, we say that $U$ is {\it Lipschitz contractible to $x_0$ in} $V$ and that $h$ is a Lipschitz contraction from $U$ to $x_0$ in $V$.

\begin{definition}[\cite{Y}, {cf.\! \cite{MY}}] \label{def:LLC} \upshape
Let $X$ be a metric space. 
We say that $X$ is {\it locally Lipschitz contractible}, for short LLC, if for any $x \in X$ and any $r > 0$, there exists $r' \in (0,r]$ such that $U(x,r')$ is Lipschitz contractible to $x$ in $U(x,r)$.
This property is also called the LLC-condition. 
\end{definition}
Here, $U(z,s)$ always denotes the open metric ball centered at $z$ with radius $s$ in a metric space.
Notice that we use a version of the definition of LLC-condition reformulated in \cite{MY}.
Obviously, if a metric space $X$ is covered by open sets $X_i$ such that each open set $X_i$ is LLC, then $X$ is LLC. 

We introduce a notion which seems to be weaker than the LLC-condition. 
\begin{definition} \upshape \label{def:WLLC}
We say that a metric space $X$ is {\it weakly locally Lipschitz contractible}, for short WLLC, if for any $x \in X$ and any open set $U \subset X$ with $x \in U$, there exists an open set $V \subset X$ with $x \in V \subset U$ such that $V$ is Lipschitz contractible in $U$ to some point of $U$. 
\end{definition}

Obviously, every LLC space is WLLC. 
Further, we have 
\begin{lemma} \label{lem:WLLC to LLC} 
Let $U$ be a subset of a metric space $V$. 
Suppose that $U$ is Lipschitz contractible in $V$. 
Then, for any $x \in U$, $U$ is Lipschitz contractible to $x$ in $V$. 

In particular, every WLLC metric space is LLC. 
\end{lemma}
\begin{proof}
By the assumption, there is a Lipscihtz homotopy $h : U \times [0,1] \to V$ such that $h_0 = \mathrm{id}_U$ and $h_1$ is a constant map. 
The image of $h_1$ is denoted by $y \in V$. 
Further, let us take an arbitrary point $x \in U$. 
Then, we consider a map $k : U \times [0,1] \to V$ defiend by 
\[
k(z,t) = \left\{ 
\begin{aligned}
& h(z, 2 t) && \text{if } t \le 1/2, \\
& h(x, 2- 2 t) && \text{if } t \ge 1/2
\end{aligned}
\right.
\]
for $z \in U$ and $t \in [0,1]$. 
This is well-defined. 
Indeed, $h(z,1) =y= h(x,1)$.
Further, $k$ is Lipschitz. 
Actually, for any $z \in U$ and $s,t \in [0,1]$ with $s \le 1/2 \le t$, we have 
\begin{align*}
d(k(z,s), k(z,t)) &= d(h(z,2 s), h(x,2-2t)) \\
&\le d(h(z,2s),h(z,1)) + d(h(x,1), h(x,2-2t)) \\
&\le \Lip(h) (1-2s) + \Lip(h) (2 t -1) \\
&= 2 \Lip(h) (t-s). 
\end{align*}
For $s, t \in [0,1]$ with $s, t \le 1/2$ or with $s, t \ge 1/2$, we also have 
\[
d(k(z,s), k(z,t)) \le 2 \Lip(h) |s-t|.
\]
For any $z,w \in U$ and $s \in [0,1]$, we obtain 
\[
d(k(z,s),k(w,s)) \le \Lip(h) d(z,w). 
\]
By the definition, $k_1$ is a constant map of the value $x$. 
Therefore, $U$ is Lipschitz contractible to $x$ in $V$. 

Let us take a WLLC metric space $X$. 
Then, for any $x \in X$ and $r > 0$, there exist an $r'>0$ such that $U(x,r')$ is Lipschitz contractible in $U(x,r)$. 
By the former statement, $U(x,r')$ is also Lipschitz contractible to $x$ in $U(x,r)$. 
Hence, we know that $X$ is LLC. 
This completes the proof. 
\end{proof}

\begin{proposition} \label{prop:LLC is preserving}
The LLC-condition is inherited to open subsets and is preserving under locally bi-Lipschitz homeomorphisms.
Namely, if $X$ is LLC and $U$ is an open subset of $X$, and if $f : X \to Y$ is a locally bi-Lipschitz homeomorphism, then $U$ and $Y$ are LLC.
\end{proposition}
\begin{proof}
Let $X$ be an LLC metric spacea and $U$ its open set.
For any $x \in U$ and $r > 0$, we can take $r' > 0$ such that $U(x,r') = U \cap U(x,r') \subset U \cap U(x,r)$. 
Here, $U(z,s)$ denotes the open ball in the whole space $X$.
Since $X$ is LLC, there is an $r'' > 0$ such that $U(x,r'')$ is Lipschitz contractible in $U(x,r')$. 
Hence, $U$ is LLC. 

Let $f : X \to Y$ be a locally bi-Lipschitz homeomorphism. 
Hence, for any $x \in X$, there is an open neighborhood $U$ of $x$ such that $f|_{U} : U \to f(U)$ is a bi-Lipschitz homeomorphism.
By the former statement, $U$ is LLC. 
Therefore, we may assume that $f$ itself is bi-Lipschitz. 
Let $C \ge 1$ be a constant which bounds $\Lip(f)$ and $\Lip(f^{-1})$. 
We prove that $Y$ is LLC. 
For any $x \in X$ and $r > 0$, we have 
\[
f(U(x, C^{-1}r)) \subset U(f(x), r) \subset f(U(x,Cr)).
\]
Since $X$ is LLC, there exist an $r' > 0$ and a Lipschitz contraction $h : U(x, Cr') \times [0,1] \to U(x, C^{-1}r)$ to $x$. 
We consider a map $k : f(U(x,Cr')) \times [0,1] \to U(f(x), r)$ defined by 
\[
k(f(y),t) = f(h(y,t))
\]
for $y \in U(x,Cr')$ and $t \in [0,1]$. 
By the construction, $k$ gives a Lipschitz contraction from $U(f(x),r')$ to $f(x)$ in $U(f(x), r)$. 
This completes the proof. 
\end{proof}

In \cite{MY}, the author and Yamaguchi defined a notion stronger than the LLC-condition as follows.

\begin{definition}[\cite{MY}] \label{def:SLLC} \upshape
A metric space $X$ is said to be {\it strongly locally Lipschitz contractible}, for short SLLC, if for any $x \in X$, there exist $r > 0$ and a Lipschitz contraction $h : U(x,r) \times [0,1] \to U(x,r)$ to $x$ such that $d(x,h(y,t))$ is monotone nonincreasing in $t$ for every $y \in U(x,r)$.
Such an $h$ is called a strong Lipschitz contraction. 
\end{definition}

It is clear that any SLLC space is LLC.
Obviously, every Banach manifold is SLLC. 

\subsection{Strong Lipscihtz contractions in the sense of \cite{RS}}
The terms of strong Lipschitz contractions were used in two papers. 
One of them was in \cite{MY} as in Definition \ref{def:SLLC} and another one was introduced by Riedweg and Sch\"appi in \cite{RS} as in the following definition. 
\begin{definition}[\cite{RS}] \label{def:RS} \upshape
A metric space $X$ {\it admits locally strong Lipschitz contractions} if for any $x \in X$, there exist $r > 0$ and $\gamma > 0$ such that every subset $S \subset U(x,r)$ admits a Lipschitz contraction $\varphi : S \times [0,1] \to X$ to some point in $X$ whose Lipschitz constant is controlled as 
\begin{equation} \label{eq:RS}
d(\varphi(z,t),\varphi(z',t')) \le \gamma \mathrm{diam} (S) |t-t'| + \gamma d(z,z')
\end{equation}
for all $z,z' \in S$ and $t,t' \in [0,1]$. 
\end{definition}
In \cite{RS}, a map $\varphi$ satisfying \eqref{eq:RS} was called a strong Lipschitz contraction of $S$. 

\begin{lemma} \label{lem:LSLC to LLC}
Any metric space admitting locally strong Lipschitz contractions is (weakly) locally Lipschitz contractible. 
\end{lemma}
\begin{proof}
Let a metric space $X$ admit locally strong Lipschitz contractions. 
For any $x \in X$, there are $r > 0$ and $\gamma > 0$ such that any subset $S$ of $U(x,r)$ admits $\varphi_S$ satisfying \eqref{eq:RS}. 
Let us take $r'' < r' \le r$. 
Then, the Lipschitz contraction $\varphi = \varphi_{U(x,r'')} : U(x,r'') \times [0,1] \to X$ satisfies 
\[
d(\varphi(z,t), \varphi(z',t')) \le 2 \gamma r'' |t-t'| + \gamma d(z,z')
\]
for all $z,z' \in U(x,r'')$ and $t,t' \in [0,1]$.
In particular, $\varphi(z,t)$ is contained in $U(x, 3 \gamma r'')$ for any $z \in U(x,r'')$ and $t \in [0,1]$. 
Hence, if $3 \gamma r'' \le r'$, then $\varphi$ is a Lipschitz contraction of $U(x,r'')$ in $U(x,r')$. 
Therefore, we know that $X$ is WLLC. 
By Lemma \ref{lem:WLLC to LLC}, $X$ is LLC.
This completes the proof. 
\end{proof}

\subsection{Examples} \label{subsec:example}
Many objects in metric geometry, related to restrictions of sectional curvature, are known to be SLLC. 
Such spaces are called CAT-spaces which are metric spaces of curvature locally bounded from above and Alexandrov spaces which have a local curvature bound from below. 
Further, Ohta introduced generalizations of CAT-spaces 
in the view point of convexities of distance functions. 
Let us recall the definitions of them, briefly. 

A metric space $X$ is geodesic if any two points $p, q$ in $X$ admit a curve $c : [0,1] \to X$ such that $c(0)=p$ and $c(1) = q$ and that $d(c(t),c(t'))=|t-t'| d(p,q)$. 
Such a curve $c$ is called a geodesic segment. 
\begin{definition}[\cite{O}] \upshape
Let $k \in (0,2]$. 
An open set $U$ of a geodesic space $X$ is called a $C_k$-{\it domain} for $k$ if for any three points $x,y,z \in U$, and any geodesic segment $c : [0,1] \to X$ between $c(0)=y$ and $c(1)=z$, we have 
\begin{equation} \label{eq:ck-domain}
d(x,c(t))^2 \le (1-t) d(x,y)^2 +t d(x,z)^2 - \frac{k}{2} t (1-t) d(y,z)^2.
\end{equation}

Let $L_1, L_2 \ge 0$. 
An open set $U$ in a geodesic space $X$ is called a $C_L$-{\it domain} for $(L_1,L_2)$ if for any three points $x,y,z \in U$, any geodesic segments $c, \bar c : [0,1] \to X$ with $c(0)=\bar c(0)=x$, $c(1)=y$, and $\bar c(1)=z$, and for all $t \in [0,1]$, we have 
\begin{equation} \label{eq:cl-domain}
d(c(t),\bar c(t)) \le \left(1+L_1 \frac{\min\{d(x,y)+d(x,z), 2 L_2\}}{2}\right) t d(y,z)
\end{equation}
\end{definition}

A $C_2$-domain is usually called a CAT(0)-domain. 
If $U$ satisfies the opposite inequality of \eqref{eq:ck-domain} for $k=2$, we say that $U$ has {\it nonnegative curvature} (in the sense of Alexandrov). 
In a complete Riemannian manifold, a CAT(0)-domain (or a nonnegatively curved domain) actually has the nonpositive (or nonnegative, respectively) sectional curvature. 
In the same way, there are definitions of synthetic sectional curvature bound from above (and from below) by a real number $\kappa$, for metric spaces, in terms of geodesic triangle comparison. 
See details \cite{BBI}. 
Ohta's results (Corollaries 2.4 and 3.2, Lemma 2.9 and Proposition 3.1 in \cite{O}) and thier proofs implies 
\begin{proposition} \label{lem:cat}
Both a $C_k$-domain and a $C_L$-domain in a geodesic metric space are SLLC, where $k \in (0,2]$ and $L_1, L_2 \ge 0$ are arbitrary. 
In particular, a CAT-space is SLLC. 
\end{proposition}
Let us give a proof only for CAT(0)-domains, for convenience. 
\begin{proof}
Let $U$ be a CAT(0)-domain in a geodesic space and $x \in U$. 
For $R > 0$, any geodesic joining two points in the ball $U(x,R)$ is contained in $U(x,2 R)$. 
Taking $R$ with $U(x,2 R) \subset U$, we define a map
\[
h : U(x, R) \times [0,1] \to U(x,R)
\]
by $h(y,t) = c_y(t)$. 
Here, $c_y : [0,1] \to X$ is a geodesic with $c_y(0)=y$ and $c_y(1) = x$. 
Using the condition \eqref{eq:ck-domain} for $k=2$ twice, we have 
\begin{equation*} 
d(h(y,t), h(z,t))^2 \le (1-t)^2 d(y,z)^2. 
\end{equation*}
Therefore, we obtain 
\[
d(h(y,t), h(z,s)) \le d(y,z) + R |t-s|.
\]
Hence, $h$ is a Lipschitz contraction to $x$.
Further, by the definition, it is a strong Lipschitz contraction in the sense of Definition \ref{def:SLLC}.
This completes the proof.
\end{proof}
Further, in \cite{RS}, 
CAT-spaces are proved 
to admit locally strong Lipschitz contractions. 

A complete geodesic space with curvature locally bounded from below, in a synthetic sense, is called an Alexandrov space. 
Any complete Riemannian manifold and the Gromov-Hausdorff limit of manifolds having a uniform lower sectional curvature bound are Alexandrov spaces. 
See details \cite{BBI}, \cite{BGP}. 
A main result in \cite{MY} states that every finite dimensional Alexandrov space is SLLC. 
The proof of this fact is not trivial. 
For instance, we can observe that the same proof of Lemma \ref{lem:cat} does not work for a geodesic space of nonnegative curvature.
To prove the result in \cite{MY}, we used the theory of gradient flow of distance functions founded by Perelman and Petrunin \cite{PP}, \cite{Pet}. 

The author do not know whether arbitrary (finite dimensional) Alexandrov space admits locally strong Lipschitz contractions. 

\subsection{Lipschitz extensions}

We recall McShane-Whitney's Lipschitz extension theorem. 
\begin{theorem}[\cite{Mc}, \cite{Wh}] \label{thm:extension}
Let $X$ be a metric space and $A$ a subset of $X$. 
Let $f : A \to \mathbb R$ be an $L$-Lipschitz function. 
Then, the following functions 
\begin{align*} 
X \ni x &\mapsto \inf_{a \in A}\, (f(a) + L d(x,a)) \text{ and } \label{eq:extension}\\
X \ni x &\mapsto \sup_{a \in A}\, (f(a) - L d(x,a)) 
\end{align*}
are $L$-Lipschitz on $X$ and extensions of $f$.
\end{theorem}
We will often use this theorem in the present paper.
The set of all real-valued Lipschitz functions on a metric space $Y$ is denoted by $\Lip(Y)$.
Using the above Lipschitz extension theorem, we have 

\begin{lemma} \label{lem:extension}
Let $A$ be a compact set in a metric space $X$. 
If a sequence $f_j \in \Lip(A)$ converges to $f \in \Lip(A)$ as $j \to \infty$ pointwise on $A$ with $\sup_j \Lip(f_j) < \infty$, then there are Lipschitz extensions $\tilde f_j$ of $f_j$ to $X$ which converges to some Lipschitz extension $\tilde f$ of $f$ pointwise on $X$ as $j \to \infty$. \end{lemma}
\begin{proof}
Let $f_j$ converge to $f$ in $\Lip(A)$ as $j \to \infty$ pointwise on $A$, with $L := \sup_j \Lip(f_j) < \infty$. 
Theorem \ref{thm:extension} ensures that the functions $\tilde f_j$ and $\tilde f$ defined as 
\[
\tilde f_j(x) = \min_{a \in A} f_j(a) + L d(a,x) \text{ and } 
\tilde f(x) = \min_{a \in A} f(a) + L d(a,x)
\]
are Lipschitz extensions of $f_j$ and $f$ to $X$.
Let us fix $x \in X \setminus A$. 
For each $j$, we take $y_j, y \in A$ such that 
\[
\tilde f_j(x) = f_j(y_j) + L d(y_j,x) \text{ and }
\tilde f(x) = f(y) + L d(y,x).
\]
Then, we have 
\[
\tilde f_j(x) \le f_j(y) + L d(y,x)
\]
for every $j$. 
Hence, we obtain $\varlimsup_{j \to \infty} \tilde f_j(x) \le \tilde f(x)$. 
Further, a subsequence $\{y_{k(j)}\}_j$ of $\{y_j\}_j$ may converge to some $y_\infty \in A$ as $j \to \infty$. 
Then, we have 
\[
\tilde f(x) \le f(y_\infty) + L d(y_\infty,x) = \lim_{j \to \infty} f_j(y_{k(j)}) + L d(y_{k(j)},x).
\]
Therefore, $\varliminf_{j \to \infty} \tilde f_j(x) \ge \tilde f(x)$.
This completes the proof.
\end{proof}

Let $\Lipb(X)$ denotes the set of all real-valued bounded Lipschitz functions on $X$. 
\begin{lemma} \label{lem:coflabby}
Let $A$ be a subset of a metric space $X$. 
Then, the maps $\Lip(X) \to \Lip(A)$ and $\Lipb(X) \to \Lipb(A)$ assigning the function restricted to $A$ are surjective. 
\end{lemma}
\begin{proof}
By Theorem \ref{thm:extension}, the restriction map $\Lip(X) \to \Lip(A)$ is surjective. 
Let $f \in \Lipb(A)$, we have a Lipschitz extension $\tilde f$ of $f$ to $X$ by using Theorem \ref{thm:extension}. 
Let $C > 0$ satisfy that $-C \le f(a) \le C$ for all $a \in A$. 
Then, the function $\min\{C, \max\{-C, \tilde f\}\}$ is bounded Lipschitz on $X$ and an extension of $f$. 
Hence, the map $\Lipb(X) \to \Lipb(A)$ is surjective. 
\end{proof}

\subsection{Locally-Lipshitz homotopy invariance of the singular Lipschitz homology}

For a metric space $X$, its singular Lipscihtz chain complex $S_\bullet^\Lip(X)$ was defiend in the introduction. 
As for the usual singular chain complex $S_\bullet(X)$, the complex $S_\bullet^\Lip(X)$ has a canonical augmentation defiend as 
\begin{equation} \label{eq:std aug} 
\varepsilon : S_0^\Lip(X) \ni \sum_{i=1}^N a_i x_i \mapsto \sum_{i=1}^\N a_i \in \mathbb Z,
\end{equation}
where $x_i \in X$ and $a_i \in \mathbb Z$. 
For the extended chain complex 
\[
\cdots \to S_k^\Lip(X) \to S_{k-1}^\Lip(X) \to \cdots \to S_0^\Lip(X) \xrightarrow{\varepsilon} \mathbb Z, 
\]
its homology is denoted by $\tilde H_\ast^\Lip(X)$, and is called the reduced singular Lipschitz homology of $X$. 

\begin{lemma} \label{lem:homotopy L} 
Let $h : X \times [0,1] \to Y$ be a locally Lipschitz homotopy. 
Then the induced homomorphisms between the (reduced) singular Lipschitz homologies coincide as follows. 
\begin{align*}
h_0{}_\ast = h_1{}_\ast &: H_k^\Lip(X) \to H_k^\Lip(Y), \\
h_0{}_\ast = h_1{}_\ast &: \tilde H_0^\Lip(X) \to \tilde H_0^\Lip(Y). 
\end{align*}
\end{lemma}
\begin{proof}
Recall that the continuous homotopy $h$ satisfies $h_0{}_\ast = h_1{}_\ast$ as a map between the usual (reduced) singular homologies. 
The proof of this fact is done by giving a chain homotopy equivalence between chain maps $h_0{}_\#$ and $h_1{}_\#$ from $S_\bullet(X)$ to $S_\bullet(Y)$. 
This chain homotopy is constructed by a decomposition of the prism $\triangle^k \times [0,1]$ into simplices. 
Such a prism decomposition is given by a combinatorial or a piecewisely linear way. 
Therefore, the standard chain homotopy gives a chain homotopy equivalence between $h_0{}_\#$ and $h_1{}_\#$ as chain maps from $S_\bullet^\Lip(X)$ to $S_\bullet^\Lip(Y)$. 
Hence, the maps $h_0{}_\ast$ and $h_1{}_\ast$ from $H_k^\Lip(X)$ to $H_k^\Lip(Y)$ are the same. 
The reduced version is proved by a similar way. 
\end{proof}

\section{Chain complices consisting of metric currents} \label{sec:current}

In this section, we denote by $X$ and $Y$ metric spaces. 
We recall the notion of currents in metric spaces introduced by Ambrosio and Kirchheim \cite{AK}. 
Here, we note that we will use a slightly modified definition from the original one. 
For the reason why we use the modified definition, see Remark \ref{rem:tight}.

\subsection{Basics of measure theory}
Let us denote by $\mu$ a Borel measure on a metric space $X$.
The {\it support} $\spt(\mu)$ of $\mu$ is defined by
\[
\spt(\mu) = \{x \in X \mid \mu(U(x,r)) > 0 \text{ for any } r > 0\}
\]
which is a closed subset of $X$. 
The measure $\mu$ is said to be finite if $\mu(X) < \infty$. 
The outer measure obtained from $\mu$ is denoted by the same symbol as $\mu$. 

We say that $\mu$ is {\it concentrated} on a subset $A$ of $X$ if $\mu(X \setminus A) = 0$.
It is 
known that if $\mu$ is concentrated on a separable set, then $\mu$ is concentrated on its support.
If $\mu$ is finite and is concentrated on a separable set, then its support is separable.

We say that $\mu$ is {\it tight} if for any $\varepsilon > 0$, there exists a compact subset $K \subset X$ such that $\mu(X \setminus K) < \varepsilon$. 
If $\mu$ is finite, then $\mu$ is tight if and only if $\mu$ is concentrated on a $\sigma$-compact set. 
In this case, $\mu$ is concentrated on its support. 
Further, $\Lipb(X)$ is dense in $L^1(X,\mu)$ if $\mu$ is a finite tight Borel measure on $X$.

Let $\mathcal M$ be a family of finite Borel measures on $X$. 
The infimum $\bigwedge_{\nu \in \mathcal M} \nu$ of $\mathcal M$ is given by 
\[
\bigwedge_{\nu \in \mathcal M} \nu (B) := 
\inf \left\{\sum_{j=1}^\infty \mu_j (B_j) \right\}
\]
for all Borel sets $B \subset X$, 
where the infimum runs over 
among 
all countable family $\{\mu_j\}_{j=1}^\infty \subset \mathcal M$ and all Borel partitions $\{B_j\}$ of $B$. 
Here, a Borel partition $\{B_j\}$ of $B$ is a disjoint countable family consisting of Borel sets satisfying $\bigcup_j B_j = B$.
By the definition, $\bigwedge_{\nu \in \mathcal M} \nu(B) \le \nu'(B)$ for any $\nu' \in \mathcal M$ and Borel set $B \subset X$.
In particular, $\mathcal M$ is a finite Borel measure. 
Further, if some $\nu' \in \mathcal M$ is tight, then $\bigwedge_{\nu \in \mathcal M} \nu$ is tight.

For another metric space $Y$ with a measurable map $f : X \to Y$, we denote by $f_\# \mu$ the push-forward measure of $\mu$ by $f$ which is defined by $f_\# \mu(B) = \mu(f^{-1}(B))$ for all Borel sets $B \subset Y$.

\subsection{Metric currents} 
From now on, $k$ denotes a nonnegative integer.

We set $\D^0(X) = \Lip_{\mathrm b}(X)$ and $\D^k(X) = \Lip_{\mathrm b}(X) \times [\Lip(X)]^k$ for $k \ge 1$.
We will often abbreviate an element $(f,\pi_1,\cdots,\pi_k) \in \D^k(X)$ by $(f,\pi)$.

\begin{definition}[\cite{AK}] \upshape \label{def:current}
A multilinear functional $T : \D^k(X) \to \mathbb R$ is called a $k$-{\it current} in $X$ if it satisfies the following three conditions. 
\begin{itemize}
\item[(1)] $T$ is {\it continuous} in the following sense. 
Let $f \in \Lipb(X)$ and $\pi_i^j \in \Lip(X)$ where $i \in \{1, \dots, k\}$ and $j \in \mathbb N$ with $\sup_{i,j} \Lip(\pi_i^j) < \infty$. 
If $\pi_i^j$ converges to some function $\pi_i$ as $j \to \infty$ pointwise on $X$ for each $i$, then we have $\lim_{j \to \infty} T(f,\pi^j) \to T(f,\pi)$, where $\pi^j = (\pi_1^j, \dots, \pi_k^j)$ and $\pi = (\pi_1, \dots, \pi_k)$.
\item[(2)] $T$ satisfies the {\it locality} as follows. 
For $(f,\pi) \in \D^k(X)$, if $\pi_i$ is constant on $\{f \neq 0\}$ for some $i$, then $T(f,\pi) = 0$. 
\item[(3)] $T$ has {\it finite mass} in the following sense. 
There is a finite tight Borel measure $\mu$ on $X$ such that 
\[
|T(f,\pi)| \le \prod_{i=1}^k \Lip(\pi_i) \int_X |f| \, d \mu
\]
holds, for all $(f,\pi) \in \D^k(X)$. 
Here, when $k=0$, the value $\prod_{i=1}^k \Lip(\pi_i)$ is regarded as $1$. 
\end{itemize}
The minimal measure of $\mu$'s satisfying the finite mass axiom (3) as above for $T$ is called the {\it mass measure} of $T$ and is denoted by $\|T\|$.
We say that a current $T$ has compact support if $\|T\|$ has compact support, or equivalently, $\|T\|$ is concentrated on a compact set. 
The set of all $k$-currents in $X$ is denoted by $\M_k(X)$ and its subset consisting of currents having compact support is denoted by $\M_k^\cpt(X)$. 
\end{definition}

Typical and essential examples of currents are as follows.
\begin{example} \upshape \label{ex}
For an $L^1$-function $\theta$ on $\mathbb R^k$ in the Lebesgue measure $\mathcal L^k$, a $k$-current $\jump{\theta} \in \M_k(\mathbb R^k)$ is defined as 
\[
\jump{\theta} (f, \pi) = \int_{\mathbb R^k} \theta f \det \left(\frac{\pa \pi_i}{\pa x_j}\right) \, d \mathcal L^k(x)
\]
for $(f,\pi) \in \D^k(\mathbb R^k)$. 
Here, $\pa \pi_i / \pa x_j$ are defined for almost everywhere $\mathbb R^k$ and are bounded integrable functions, due to Rademacher's theorem. 
\end{example}

There is a useful characterization of the mass measures of currents as follows. 
\begin{proposition}[Proposition 2.7 in \cite{AK}] \label{prop:mass}
Let $T$ be a metric $k$-current in $X$. 
For every Borel set $B$ in $X$, we have 
\[
\|T\|(B) = \sup \sum_{j=1}^\infty T(\chi_{B_j}, \pi^j),
\]
where the supremum runs among all Borel partitions $\{B_j\}$ of $B$ and all Lipschitz maps $\pi^j = (\pi^j_i)_{1 \le i \le k}: X \to \mathbb R^k$ with $\Lip(\pi_i^j) \le 1$ for all $i$. 
\end{proposition}

For $k$-currents $T, S$ in $X$ and a Borel set $B$ in $X$, we have 
\begin{align*}
&\|T+S\|(B) \le \|T\|(B)+\|S\|(B), \\
&\|-T\|(B) = \|T\|(B).
\end{align*}
These properties follow from the definition. 
The second property induces that $\spt(-T)=\spt(T)$. 

We recall fundamental operations to obtain currents. 
From now on, let $T \in \M_k(X)$. 
For another metric space with a Lipschitz map $\phi : X \to Y$, we have the {\it push-forward} $\phi_\# T$ of $T$ by $\phi$ defined by 
\[
\phi_\# T (f, \pi) = T(f \circ \phi, \pi \circ \phi)
\]
for $(f,\pi) \in \D^k(Y)$, which is a $k$-current in $Y$. 
Further, we have $\|\phi_\# T\| \le \Lip(\phi)^k \phi_\# \|T\|$ as measures on $Y$. 
If $\phi$ is a bi-Lipschitz embedding, then by Lemma \ref{lem:coflabby}, $\phi_\# : \M_k(X) \to \M_k(Y)$ is injective. 
Since $\Lipb(X)$ is dense in $L^1(X,\|T\|)$, the $k$-current $T$ can be extended to a multilinear functional on $L^1(X,\|T\|) \times [\Lip(X)]^k$ in a unique way. 
Therefore, for any Borel set $A \subset X$, the functional $T \lfloor A$ is defined by 
\[
T \lfloor A(f,\pi) = T(\chi_A f, \pi)
\]
for $(f,\pi) \in \D^k(X)$, where $\chi_A$ denotes the characteristic function of $A$, which is called the {\it restriction} of $T$ to $A$ and is also a $k$-current in $X$.
\begin{lemma} \label{lem:spt}
Let $T$ be a current in $X$ and $A$ a Borel set in $X$.
The mass measure of $T \lfloor A$ coincides with the restricted measure $\|T\| \lfloor A$ to $A$ defined as 
\[
\|T\| \lfloor A (B) := \|T\|(A \cap B)
\]
for every Borel set $B \subset X$. 
In particular, the support of $T \lfloor A$ is contained in $\spt(T) \cap \bar A$, where $\bar A$ is the closure of $A$ in $X$.
\end{lemma}
\begin{proof}
We take a metric current $T$ in $X$ of degree $k$ and a Borel set $A$. 
By Proposition \ref{prop:mass}, 
\[
\|T \lfloor A\| (B) = \sup \sum_{j=1}^\infty T (\chi_A \chi_{B_j}, \pi^j)
\]
holds, where the supremum runs among all Borel partitions $\{B_j\}$ of $B$ and all $k$-tuples $\pi^j = (\pi_i^j)_{1 \le i \le k}$ of $1$-Lipschitz functions $\pi_i^j : X \to \mathbb R$.
Since $\chi_{A \cap B_j} = \chi_A \chi_{B_j}$ and a Borel partition $\{B_j\}$ of $B$ gives a Borel partition $\{B_j \cap A\}$ of $A \cap B$, 
the value $\|T\|(A \cap B)$ actually coincides with $\|T \lfloor A\|(B)$, by Proposition \ref{prop:mass} again. 
Therefore, we have $\|T \lfloor A\| = \|T\| \lfloor A$. 
Hence, we obtain $\spt(T \lfloor A) = \spt (\|T\| \lfloor A) \subset \spt \|T\| \cap \bar A = \spt\, T \cap \bar A$. 
This completes the proof. 
\end{proof}

It is clear that $T \lfloor \spt(T) = T$ as currents in $X$. 
The restriction is a current in the whole space $X$ in general. 
However, if $T$ has compact support, $T$ itself can be regard as a current in its support as follows. 

\begin{lemma} \label{lem:cpt spt}
Let $T \in \M_k^\cpt(X)$. 
Then, there is a unique $k$-current $T'$ in $\spt(T)$ such that $T = \iota_\# T'$, where $\iota : \spt(T) \to X$ is the inclusion.
\end{lemma}
\begin{proof}
Let us take $T \in \M_k^\cpt(X)$ and set $A = \spt(T)$. 
A functional $T' : \D^k(A) \to \mathbb R$ is defined as follows. 
For $(f,\pi) \in \D^k(A)$, by Lemma \ref{lem:coflabby}, we have an extension $(\tilde f, \tilde \pi) \in \D^k(X)$ of $(f,\pi)$. 
Then, we set 
\[
T'(f,\pi) = T(\tilde f, \tilde \pi).
\] 
First, we check that this value is independent on the choice of $(\tilde f, \tilde \pi)$ and show that $T'$ is multilinear and satisfies the locality. 
Let $\hat f \in \Lipb(X)$ be another extension of $f$. 
By the finite mass axiom, we have $T(\tilde f, \tilde \pi) = T(\hat f, \tilde \pi)$. 
Further, if $\hat f$ is a bounded Borel function on $X$ with $\hat f |_A = f$, then $T(\hat f, \tilde \pi) = T(\tilde f, \tilde \pi)$, because $T$ can be regarded as a functional on $L^1(X,\|T\|) \times [\Lip(X)]^k$. 
This implies the linearity of $T'(f, \pi)$ in $f$. 
Here, recall that $T$ as the functional on $L^1(X,\|T\|) \times [\Lip(X)]^k$ satisfies the strengthened locality, continuity and finite mass axiom, as stated in \cite{AK}. 
For another Lipschitz extension $\hat \pi \in [\Lip(X)]^k$ of $\pi$, the strengthened locality of $T$ implies $T(\tilde f, \tilde \pi) = T(\tilde f, \hat \pi)$. 
Further, the strengthened locality of $T$ implies the multilinearity of $T'(f,\pi)$ in $\pi$ and the locality of $T'$. 

Let us define a finite Borel measure $\mu$ on $A$ by $\mu(B) = \|T\|(B)$ for any Borel set $B \subset A$. 
The tightness of $\|T\|$ ensures that $\mu$ is tight. 
For any $(f, \pi) \in \D^k(A)$, we have 
\[
|T'(f,\pi)| = |T(\tilde f, \tilde \pi)| \le \prod_{i=1}^k \Lip(\tilde \pi_i) \int_X |\tilde f| \, d \|T\| = \prod_{i=1}^k \Lip(\pi_i) \int_X |f| \, d \mu. 
\]
Here, we note that a Lipschitz extension $\tilde \pi_i$ of $\pi_i$ can be chosen as $\Lip(\tilde \pi_i) = \Lip(\pi_i)$. 
Hence, $T'$ has finite mass. 

Finally, we prove that $T'$ is continuous. 
Let $\pi_i^j \in \Lip(A)$, $i =1, \dots,k$, $j \in \mathbb N$ with $\sup_{i,j} \Lip(\pi_i^j) \le L < \infty$ such that $\pi_i^j$ converges to $\pi_i$ as $j \to \infty$ pointwise on $A$, for every $i$.  
By Lemma \ref{lem:extension}, we have extensions $\tilde \pi_i^j$ of $\pi_i^j$ to $X$ with $\sup_{i,j} \Lip(\tilde \pi_i^j) \le L$, such that $\tilde \pi_i^j$ converges to some Lipschitz extension $\tilde \pi_i$ of $\pi_i$ pointwise on $X$ as $j \to \infty$, for each $i$. 
By the continuity of $T$, we have 
\[
\lim_{j \to \infty} T'(f, \pi^j) = \lim_{j \to \infty} T(\tilde f, \tilde \pi^j) = T(\tilde f, \tilde \pi) = T'(f,\pi).
\]
Hence $T$ is continuous. 
Therefore, $T' \in \M_k^\cpt(A)$. 
By the definition, $\iota_\# T' = T$. 
Since $\iota_\#$ is injective, such a $T'$ is unique. 
This completes the proof.
\end{proof}

As a corollary to Lemma \ref{lem:cpt spt}, 
the push-forward operator $\phi_\# : \M_k^\cpt(X) \to \M_k^\cpt(Y)$ is also defined, for a locally Lipschitz map $\phi : X \to Y$, as follows.
For $T \in \M_k^\cpt(X)$, taking $T'$ as in Lemma \ref{lem:cpt spt} for $T$, and set $\phi_\# T := \iota_\# (\phi |_{\spt(T)})_\# T' \in \M_k^\cpt(Y)$, where $\iota : \phi(\spt(T)) \to Y$ is the inclusion. 
By the construction, the pushforwards have the following functorial property: 
\[
(\psi \circ \phi)_\# = \psi_\# \phi_\# : \M_k^\cpt(X) \to \M_k^\cpt(Z)
\]
for locally Lipschitz mappings $\phi : X \to Y$ and $\psi : Y \to Z$. 

The {\it boundary} $\pa T$ of $T$ is a functional on $\D^{k-1}(X)$ defined by 
\[
\pa T(f,\pi) = T(1, f, \pi).
\]
It satisfies the continuity and locality, but does not have finite mass in general. 
It is trivial that 
\begin{equation} \label{eq:comm}
\partial \phi_\# = \phi_\# \partial
\end{equation}
holds on $\M_k^\cpt(X)$, for a locally Lipschitz map $\phi : X \to Y$. 
By the locality, we have 
\begin{equation} \label{eq:dd=0}
\pa \pa T = 0.
\end{equation} 

\begin{definition}[\cite{AK}] \label{def:normal} \upshape 
A $k$-current $T$ in $X$ is {\it normal} if its boundary has finite mass, or equivalently, $\pa T \in \M_{k-1}(X)$. 
Here, when $k = 0$, we always regard $T$ as normal and $\pa T = 0$.
We denote the set of all normal $k$-currents in $X$ by $\N_k(X)$ and set $\N_k^\cpt(X) = \N_k(X) \cap \M_k^\cpt(X)$.
\end{definition}

By the definition, \eqref{eq:dd=0} and \eqref{eq:comm}, 
the group $\N_\bullet(X) = \bigoplus_{k=0}^\infty \N_k(X)$ becomes a chain complex with the boundary map $\partial$ and $\phi_\# : \N_\bullet(X) \to \N_\bullet(Y)$ is a chain map for any Lipschitz map $\phi : X \to Y$. 
Further, by Slicing Theorem 5.6 in \cite{AK}, if $T \in \N_k(X)$, then we have 
\begin{equation} \label{eq:boundary}
\spt(\partial T) \subset \spt(T).
\end{equation}
Hence, $\N_\bullet^\cpt(X) = \bigoplus_{k=0}^\infty \N_k^\cpt(X)$ is a subcomplex of $\N_\bullet(X)$.
Further, by \eqref{eq:comm}, $\phi_\# : \N_\bullet^\cpt(X) \to \N_\bullet^\cpt(Y)$ is well-defined as a chian map, for any {\it locally} Lipschitz map $\phi : X \to Y$. 

\begin{lemma} \label{lem:N cpt spt}
Let $T \in \N_k^\cpt(X)$. 
Then, there is a unique $T' \in \N_k^\cpt(\mathrm{spt}\, T)$ such that $\iota_\# T' = T$, where $\iota : \mathrm{spt}\, T \to X$ is the inclusion. 
\end{lemma}
\begin{proof}
Let us take $T \in \N_k^\cpt(X)$. 
By Lemma \ref{lem:cpt spt}, there is a unique $T' \in \M_k^\cpt(K)$ such that $\iota_\# T' = T$, where $K = \mathrm{spt}\, T$.
On the other hands, since $\pa T \in \M_{k-1}^\cpt(X)$, by using Lemma \ref{lem:cpt spt} again, we have $S \in \M_{k-1}^\cpt(K)$ such that $\iota_\# S = \pa T$.
Then, we obtain 
\[
\iota_\# S = \partial T = \pa \iota_\# T' = \iota_\# \partial T'.
\]
Since $\iota_\#$ is injective from Lemma \ref{lem:coflabby}, we conclude that $\pa T' = S$ and hence $\|\pa T'\|=\|S\|$.
Therefore, $T' \in \N_k^\cpt(K)$. 
This completes the proof. 	
\end{proof}

\subsection{Integral currents}
For a metric space $X$, we denote by $\mathcal H^k = \mathcal H_X^k$ the $k$-dimensional Hausdorff measure on $X$. 
A subset $S$ of $X$ is called a {\it countably $\mathcal H^k$-rectifiable set} if there are countably many Borel subsets $B_j$ of $\mathbb R^k$ and Lipschitz maps $\phi_j : B_j \to X$ such that 
\[
\mathcal H^k \left(S \setminus \bigcup_{j=1}^\infty \phi_j(B_j) \right) = 0.
\]

\begin{definition}[\cite{AK}] \upshape
A $k$-current $T$ in $X$ is said to be {\it rectifiable} if $\|T\|$ is concentrated on a countably $\mathcal H^k$-rectifiable set and is absolutely continuous in $\mathcal H^k$. 
When $k=0$, a rectifiable $0$-current $T$ is represented as 
\[
T(f) = \sum_{j=1}^\infty \theta_j f(x_j)
\]
for $f \in \Lipb(X)$, where $\theta_j \in \mathbb R$ and $x_j \in X$. 
Such a $T$ is written as 
\begin{equation} \label{eq:0-rect}
T = \sum_{j=1}^\infty \theta_j [x_j].
\end{equation} 

When $k \ge 1$, a $k$-current $T$ in $X$ is called an {\it integer rectifiable} current if it is rectifiable and for any open set $O \subset X$ and a Lipschitz function $\phi$, there is an integral valued $L^1$-function $\theta$ on $\mathbb R^k$ such that 
\[
\phi_\# (T \lfloor O) = \jump{\theta}
\]
as currents in $\mathbb R^k$. 
An {\it integer rectifiable $0$-current} $T$ in $X$ is defined as \eqref{eq:0-rect} such that $\theta_j \in \mathbb Z$ for all $j$.
\end{definition}
For integral rectifiable $0$-current $T$ represented as \eqref{eq:0-rect}, the families $\{x_j\}$ and $\{\theta_j\}$ can be finite sets. 

\begin{definition}[\cite{AK}] \upshape 
A $k$-current in $X$ is said to be {\it integral} if it is integer rectifiable and normal. 
We denote by $\I_k(X)$ the set of all integral $k$-currents in $X$ and set $\I_k^\cpt(X) = \I_k(X) \cap \M_k^\cpt(X)$. 
\end{definition}

The Boundary Rectifiability Theorem 8.6 in \cite{AK} says that $\I_\bullet(X) = \bigoplus_{k=0}^\infty \I_k(X)$ becomes a subcomplex of $\N_\bullet(X)$. 
The group $\I_\bullet^\cpt(X) = \bigoplus_{k=0}^\infty \I_k^\cpt(X)$ is also a subcomplex of $\I_\bullet(X)$.
As mentioned in the introduction, we will compare the homology of $\I_\bullet^\cpt(X)$, denote by $H_\ast^\IC(X)$, with the singular (Lipschitz) homology of $X$.
Further, we can check that the chain map $\phi_\# : \I_\bullet^\cpt(X) \to \I_\bullet^\cpt(Y)$ is well-defined, for a locally Lipschitz map $\phi : X \to Y$. 
Therefore, $\I_\bullet^\cpt$ can be regarded as a covariant functor from the category of all metric spaces and all locally Lipschitz maps to the category of all chain groups and all chain maps. 

\begin{lemma}\label{lem:I cpt spt}
For any $T \in \I_k^\cpt(X)$, there is a unique $T' \in \I_k^\cpt(\mathrm{spt}\, T)$ such that $\iota_\# T' = T$, where $\iota : \mathrm{spt}\, T \to X$ is the inclusion.
\end{lemma}
\begin{proof}
Let $T \in \I_k^\cpt(X)$. 
By Lemma \ref{lem:N cpt spt}, there is a unique $T' \in \N_k^\cpt(K)$ such that $\iota_\# T' = T$, where $K = \mathrm{spt}\, T$.
From the definition, there is a countably $\mathcal H^k$-rectifiable set $S \subset X$ such that $\|T\|(X \setminus S) = 0$. 
Then, we note that $S \cap K$ is also a countably $\mathcal H^k$-rectifiable set in $K$. 
Hence, we have $\|T'\|(K \setminus S)=\|T\|(X \setminus S)=0$. 
For a Borel set $B \subset K$ with $\mathcal H^k(B)=0$, we have $\|T'\|(B) = \|T\|(B)=0$. 
Therefore, $T'$ is a rectifiable current. 
We prove that $T'$ is integer rectifiable. 
For any open set $O$ in $K$, there is an open set $U$ in $X$ such that $O = K \cap U$. 
Let us take a Lipschitz map $\phi : K \to \mathbb R^k$. 
By Lemma \ref{lem:extension}, there is an extension $\psi : X \to \mathbb R^k$ such that $\psi$ is Lipschitz.
Since $T$ is integer rectifiable, there is an integrable function $\theta : \mathbb R^k \to \mathbb Z$ such that 
\[
\psi_\#(T \lfloor U) = \jump{\theta}.
\]
Now, we note that $T \lfloor U = (\iota_\# T') \lfloor U = \iota_\# (T' \lfloor O)$. 
Therefore, we obtain 
\[
\jump{\theta} = \psi_\# \iota_\# (T' \lfloor O)= \phi_\# (T' \lfloor O).
\]
This implies that $T'$ is integer rectifiable. 
This completes the proof. 
\end{proof}

\subsection{Locally Lipschitz homotopy invariance of the homology of currents}
We check that the homology $H_\ast^\IC(X)$ is independent on locally Lipschitz homotopy. 
Further, we have 
\begin{lemma} \label{lem:homotopy IC}
Let $h : X \times [0,1] \to Y$ be a locally Lipschitz homotopy. 
Then, we have $H_\ast(\N_\bullet^\cpt(h_0)) = H_\ast(\N_\bullet^\cpt(h_1))$ as morphisms from $H_\ast(\N_\bullet^\cpt(X))$ to $H_\ast(\N_\bullet^\cpt(Y))$. 
Further, we have $H_\ast^\IC(h_0) = H_\ast^\IC(h_1)$ as morphisms from $H_\ast^\IC(X)$ to $H_\ast^\IC(Y)$.
\end{lemma}

To prove Lemma \ref{lem:homotopy IC}, let us recall a product of currents considered in \cite{AK} and \cite{We}. 
Let $T$ be a $k$-current in a metric space $X$. 
Then, we define a functional $T \times [0,1] : \D^{k+1}(X \times [0,1]) \to \mathbb R$ by 
\[
(T \times [0,1]) (f, \pi) = \sum_{i=1}^{k+1} (-1)^i \int_0^1 T \left(f_t \frac{\partial \pi_i}{\partial t}, \hat \pi_i{}_t \right) dt
\]
for $(f,\pi) \in \D^{k+1}(X \times [0,1])$, 
where $\hat \pi_i{}_t = (\pi_j{}_t)_{j \neq i}$. 
Here, the partial derivative $\partial \pi_i(x,t) / \partial t$ is defined for $\mathcal L^1$-a.e.\! $t$ and $\|T\|$-a.e.\! $x$, by Rademacher's Theorem and Fubini's Theorem. 
The functional $T \times [0,1]$ does not satisfy the continuity in general. 
Ambrosio and Kirchheim, and Wenger proved 
\begin{proposition}[\cite{We}, \cite{AK}] \label{prop:product}
Let $T \in \N_k(X)$ with bounded support. 
Then, $T \times [0,1] \in \N_{k+1}(X \times [0,1])$ with boundary 
\[
\partial (T \times [0,1]) = T \times [0] - T \times [1] - (\partial T) \times [0,1]. 
\]
Here, $(T \times [t])(f, \pi) = T(f_t, \pi_t)$ for $(f,\pi) \in \D^k(X \times [0,1])$. 

In addition, if $T \in \I_k(X)$, then $T \times [0,1] \in \I_{k+1}(X \times [0,1])$. 
\end{proposition}

By the construction, if $T$ has compact support, then $T \times [0,1]$ has compact support.
Indeed, we have, for $(f,\pi) \in \D^{k+1}(X \times [0,1])$ with $\Lip(\pi_i) \le 1$,  
\[
|(T \times [0,1])(f,\pi)| \le (k+1) \int_0^1 \int_X |f_t| \, d \|T\|(x) \, d t. 
\]
Hence, the support of $T \times [0,1]$ is contained in $\spt(T) \times [0,1]$. 
\begin{proof}[Proof of Lemma \ref{lem:homotopy IC}]
Let us denote by $\mathbf C_k$ one of $\N_k^\cpt$ or $\I_k^\cpt$.
Let $h : X \times [0,1] \to Y$ be a locally Lipschitz homotopy. 
Let us define a map 
\[
P : \mathbf C_k(X) \to \mathbf C_{k+1}(Y)
\]
by 
\[
P(T) = h_\# (T \times [0,1]). 
\] 
Then, by Proposition \ref{prop:product}, we have 
\[
\partial P(T) + P(\partial T) = h_\# (T \times [0] - T \times [1]) = h_0{}_\# T - h_1{}_\# T.
\]
Hence, the map $P$ is a chain homotopy between $h_0{}_\#$ and $h_1{}_\#$. 
Therefore, the induced maps between homologies are the same. 
This completes the proof. 
\end{proof}

\subsection{Mayer-Vietoris type property}

\begin{lemma}[Localization Lemma 5.3 in \cite{AK}] \label{lem:localization}
Let $T$ be a normal $k$-current in a metric space $X$ and $f : X \to \mathbb R$ a Lipschitz function.
For almost all $s \in \mathbb R$, the restriction $T \lfloor \{f \le s \}$ of $T$ to the sublevel set of $f$ is a normal current in $X$. 
Further, if $T$ is integral, then so is $T \lfloor \{f \le s\}$ for a.e.\! $s \in \mathbb R$. 
\end{lemma}
As a corollary to Lemma \ref{lem:localization}, 
we know that $T \lfloor \{f > s\} = T -  T \lfloor \{f \le s\}$ is also a normal current, for a.e. $s \in \mathbb R$. 
Further, if $T$ has a compact support, then $T \lfloor \{f \le s\}$ and $T \lfloor \{f > s\}$ have compact support. 

\begin{lemma} \label{lem:a}
Let $U$ and $V$ be open subsets in a metric space $X$.
Then, there is an exact sequence as follows. 
\[
0 \to \N_\bullet^\cpt(U \cap V) \xrightarrow{\psi} \N_\bullet^\cpt(U) \oplus \N_\bullet^\cpt(V) \xrightarrow{\varepsilon} \N_\bullet^\cpt(U \cup V) \to 0, 
\]
where $\psi$ and $\varepsilon$ are defined as 
\[
\psi(T) = (i_\# T, -i'_\#T), \hspace{10pt} \varepsilon(S,S') = j_\# S+j_\#' S'. 
\]
Here, $i : U \cap V \to U$, $i' : U \cap V \to V$, $j : U \to U \cup V$, $j' : V \to U \cup V$ are the inclusions.

Further, if we replace $\N_\bullet^\cpt$ with $\I_\bullet^\cpt$, then we obtain an exact sequence 
\[
0 \to \I_\bullet^\cpt(U \cap V) \xrightarrow{\psi} \I_\bullet^\cpt(U) \oplus \I_\bullet^\cpt(V) \xrightarrow{\varepsilon} \I_\bullet^\cpt(U \cup V) \to 0.
\]
\end{lemma}
\begin{proof}
Let us denote by $\mathbf C_k$ one of $\N_k^\cpt$ and $\I_k^\cpt$. 
Let us take open sets $U$ and $V$ in $X$ and consider the following sequence 
\[
0 \to \mathbf C_k(U \cap V) \xrightarrow{\psi} \mathbf C_k(U) \oplus \mathbf C_k(V) \xrightarrow{\varepsilon} \mathbf C_k(U \cup V) \to 0. 
\]
Since $i_\#$ and $i'_\#$ are injective, $\psi$ is injective. 
By the definition, we have $\varepsilon \psi = 0$. 
We prove that $\varepsilon$ is surjective. 
Let us take $T \in \mathbf C_k(U \cup V)$. 
We denote by $d$ the distance function from $X \setminus V$. 
Since $\spt(T)$ is compact and is contained in $U \cup V$, there is an $r > 0$ such that 
\[
\spt(T) \cap \{d \le r\} \subset U.
\]
By Lemma \ref{lem:localization}, we can take $r$ such that $T \lfloor \{d \le r\}$ is in $\mathbf C_k$. 
Then, $S := T \lfloor \{d \le r\}$ is regarded as a current in $\mathbf C_k(U)$ and $S' := T \lfloor \{d > r\}$ can be regarded as a current in $\mathbf C_k(V)$, by Lemmas \ref{lem:spt} and \ref{lem:cpt spt}. 
By the construction, we have $T = j_\# S+j_\#' S'$. 
Hence, the map $\varepsilon$ is surjective. 
Next, we take an element $(S,S') \in \mathbf C_k(U) \oplus \mathbf C_k(V)$ with $j_\# S+ j_\#' S' = 0$. 
Since $\spt(-S) = \spt(-j_\# S) = \spt(j_\# S) = \spt(S)$, we know that $\spt(S) = \spt(S')$ and it is contained in $U \cap V$. 
By Lemma \ref{lem:cpt spt}, the current $S$ can be regarded as a current in $U \cap V$, say $T$. 
Then, we have $\psi(T) = (i_\#T,-i_\#' T) = (S,S')$. 
This completes the proof.  
\end{proof}

\subsection{A natural transformation $[\,\cdot\,]$ from $S_\bullet^\Lip$ to $\I_\bullet^\cpt$}
In the introduction, we already define the complex $S_\bullet^\Lip(X)$ of singular Lipschitz chains in a metric space $X$, which is a subcomplex of the usual integral singular chain complex $S_\bullet(X)$. 
Following \cite{RS}, we define a chain map $[\,\cdot\,] : S_\bullet^\Lip(X) \to \I_\bullet^\cpt(X)$. 
For each singular Lipschitz simplex $\sigma : \triangle^k \to X$, which is just a Lipschitz map, a $k$-current $[\sigma]$ in $X$ is defined by  
\[
[\sigma] = \sigma_\# \jump{1_{\triangle^k}}.
\]
By the definition, we have 
\[
\spt([\sigma]) \subset \mathrm{im}(\sigma). 
\]
Its $\mathbb Z$-linear extension gives a group homomorphism $[\,\cdot\,] : S_k^\Lip(X) \to \I_k^\cpt(X)$. 
Then, we have 
\[
\spt([c]) \subset \mathrm{im}(c)
\]
for every Lipschitz chain $c$. 
Here, for a singular chain $c = \sum_{\sigma} a_\sigma \sigma$, its image is defined by 
\[
\mathrm{im}(c) = \bigcup_{a_\sigma \neq 0} \mathrm{im}(\sigma). 
\]
We note that, Stokes's theorem for Lipschitz functions on $\triangle^k$ holds as the following form. 
\[
\int_{\triangle^k} \det \frac{\pa (f_1,\dots,f_k)}{\pa (s_1,\dots,s_k)} \, d \mathcal L^k(s) 
= \int_{\pa \triangle^k} f_1 \det \frac{\pa (f_2, \dots, f_k)}{\pa (t_1, \dots, t_{k-1})}\, d \mathcal L^{k-1}(t)
\]
for $f_1, \dots, f_k \in \Lip(\triangle^k)$, where $\int_{\pa \triangle^k}$ is the sum of integrations over $(k-1)$-faces of $\triangle^k$ with orientations, and $t$ is an intrinsic coordinate of each face. 
This formula is actually proved by a standard smoothing argument, and is represented as 
\[
(\pa \jump{1_{\triangle^k}})(f_1, \dots, f_k) = \sum_{i=0}^k (\iota_i {}_\# \jump{1_{\triangle^{k-1}}})(f_1, \dots, f_k)
\] 
where $\iota_i : \triangle^{k-1} \to \triangle^k$ is an orientation preserving isometric embedding into a face of $\triangle^k$.
From this formulation, the map $[\,\cdot\,] : S_\bullet^\Lip(X) \to \I_\bullet^\cpt(X)$ is known to be a chain map. 
Further, it is natural in the sense that $[\phi_\# c] = \phi_\# [c]$ for a locally Lipschitz map $\phi : X \to Y$ to another metric space $Y$ and $c \in S_\bullet^\Lip(X)$.

\subsection{The groups of $0$-chains} 
Let $X$ be a metric space. 
Obviously, the groups $S_0^\Lip(X)$ and $S_0(X)$ are the same. 
For the group $\I_0^\cpt(X)$ of integral $0$-currents, the following two lemmas hold	. 
\begin{lemma} \label{lem:0-chains}
The map $[\,\cdot\,] : S_0^\Lip(X) \to \I_0^\cpt(X)$ is an isomorphism. 
\end{lemma}
\begin{proof}
Let us take $c = \sum_{i=1}^N a_i x_i \in S_0^\Lip(X) = S_0^\Lip(X)$. 
Then, $[c] = \sum_{i=1}^N a_i [x_i] \in \I_0^\cpt(X)$ from the definition. 
Hence, the map $[\,\cdot\,] : S_0^\Lip(X) \to \I_0^\cpt(X)$ is surjective. 
We prove that it is injective. 
We assume $[c]=0$, that is, $\sum_{i=1}^N a_i f(x_i) = 0$ for every bounded Lipschitz map $f : X \to \mathbb R$. 
Fix an index $i$, we can take a function $f \in \Lipb(X)$ such that $f(x_i)=1$ and $f(x_j) = 0$ for all $j \neq i$. 
Hence, $a_i = 0$ for all $i$. 
It implies $c=0$. 
This completes the proof. 
\end{proof}

Let us consider the map $\varepsilon : \N_0^\cpt(X) \to \mathbb R$ defiend by 
\[
\varepsilon T = T(1) 
\]
for all $T \in \N_0^\cpt(X)$.
The restriction of it to the group $\I_0^\cpt(X)$ is also represented as the same symbol $\varepsilon : \I_0^\cpt(X) \to \mathbb Z$, where we note that the target group can be $\mathbb Z$. 
We have $\partial \varepsilon S = 0$ for all $S \in \N	_1^\cpt(X)$, by the locality axiom of currents. 
Hence, the both $\varepsilon$'s are augmentations of the complices $\N_\bullet^\cpt(X)$ and $\I_\bullet^\cpt(X)$.

\begin{lemma}\label{lem:augmentation}
Let $\varepsilon : S_0^\Lip(X) \to \mathbb Z$ be the standard augmentation defiend in \eqref{eq:std aug} which is denoted by the same symbol as the map $\varepsilon : \I_0^\cpt(X) \to \mathbb Z$.
Then, we have $\varepsilon \circ [\,\cdot\,] = \varepsilon$. 
\end{lemma}
\begin{proof}
The lemma follows directly from the definitions. 
\end{proof}

\begin{lemma} \label{lem:dim axiom}
Let $X_0$ be a metric space consisting of a single point.
Then, we have $H_0^\IC(X_0) \cong \mathbb Z$ and $H_k^\IC(X_0) = 0$ for $k \ge 1$. 
In addition, $H_0(\N_\bullet^\cpt(X_0)) \cong \mathbb R$ and $H_k(\N_\bullet^\cpt(X_0)) = 0$ for $k \ge 1$. 
\end{lemma}
\begin{proof}
By the definition, we have $\I_0^\cpt(X_0) \cong \mathbb Z$ and $\N_0^\cpt(X_0) \cong \mathbb R$.
By Theorem 3.9 in \cite{AK}, normal currents have the following property: 
if $T$ is a normal $k$-current in a metric space $Y$, then $\|T\|$ is absolutely continuous in the $k$-dimensional Hausdorff measure $\mathcal H^k$. 
Therefore, we have $\I_k^\cpt(X_0) = 0 = \N_k^\cpt(X_0)$ for all $k \ge 1$. 
This implies the conclusion of the lemma. 
\end{proof}

\section{A way comparing homologies by using cosheaves}\label{sec:cosheaf}
In this section, we recall the notion of cosheaves and give a technique to compare two homologies associated to cosheaves. 
Let $\Ab$ and $\cab$ denote the categories of all abelian groups with group homomorphisms and of all chain complices of abelian groups with chain maps, respectively. 
Throughout the present paper, any chain complex was and will be indexed by nonnegative integers. 
A complex $C \in \cab$ is represented as $C = C_\bullet = (C_\bullet, \pa) = (\cdots \to C_k \xrightarrow{\pa_k} C_{k-1} \to \cdots \xrightarrow{\pa_1} C_0) = \bigoplus_{k \ge 0} C_k$. 
We recall that the $m$-th homology $H_m(C_\bullet)$ of $C_\bullet$ is defined as $H_m(C_\bullet) = \mathrm{Ker}\, \pa_m / \mathrm{Im}\, \pa_{m-1}$ if $m \ge 1$ and as $H_0(C_\bullet) = C_0 / \mathrm{Im}\, \pa_1$. 
For a complex $C_\bullet = (C_\bullet, \pa) \in \cab$ and an abelian group $A \in \Ab$, a homomorphism $\varepsilon : C_0 \to A$ is called an augmentation of $C_\bullet$ if $\varepsilon \circ \partial_1 = 0$ holds, and the augmentation $\varepsilon$ is also denoted by $\varepsilon : C_\bullet \to A$. 
Let $\mathsf C$ be one of $\Ab$ and $\cab$. 
We also define the reduced homology $\tilde H_\ast(C_\bullet)$ of the complex $C_\bullet$ augmented by $\varepsilon$ as $\tilde H_m(C_\bullet) = H_m(C_\bullet)$ if $m \ge 1$ and $\tilde H_0(C_\bullet) = H_0(C_\bullet \xrightarrow{\varepsilon} A) = \mathrm{Ker}\, \varepsilon / \mathrm{Im}\, \pa_1$.

\subsection{Cosheaves} 
Cosheaves were introduced by Bredon \cite{Br}. 
Proofs of statements below about (pre)cosheaves, we refer to the book \cite{Br}. 

In this section, $X$ always denotes a topological space. 
The set $\mathsf O(X)$ of all open subsets of $X$ is regarded as a category in a usual way, that is, open sets $U \in \mathsf O(X)$ are objects and a morphism $U \to V$ uniquely exists if and only if $U \subset V$, for $U, V \in \mathsf O(X)$.

Let $\mathsf C$ denote one of the categories $\Ab$ or $\cab$. 
A {\it precosheaf} on $X$ (of $\mathsf C$-valued) is a covariant functor $\A$ from $\mathsf O(X)$ to $\mathsf C$.
We will use only a precosheaf $\A$ such that $\A(\emptyset) = 0$. 
For a precosheaf $\A$ on $X$ and $U, V \in \mathsf O(X)$, the morphism induced by $U \subset V$ is denoted by $i_{V,U} = i_{V,U}^{\,\A} : \A(U) \to \A(V)$.
A precosheaf $\A$ is said to be {\it flabby} if $i_{X,U}$ is injective for every $U \in \mathsf O(X)$, or equivalently, $i_{V,U}$ is injective for every $U$ and $V$ in $\mathsf O(X)$ with $U \subset V$. 
\begin{definition}[\cite{Br}] \upshape
A precosheaf $\A$ on $X$ is called a {\it cosheaf} if it satisfies the following: 
For any family of open sets $\mathcal U = \{U_\alpha\}_{\alpha \in A}$ of $X$, setting $U = \bigcup_{\alpha \in A} U_\alpha$, the short complex
\begin{equation} \label{eq:short}
\bigoplus_{\alpha,\beta \in A} \A(U_\alpha \cap U_\beta) \xrightarrow{\Phi} \bigoplus_{\alpha \in A} \A(U_\alpha) \xrightarrow{\varepsilon} \A(U) \to 0
\end{equation}
is exact, where $\Phi = \sum_{\alpha, \beta} i_{U_\alpha, U_\alpha \cap U_\beta} - i_{U_\beta, U_\alpha \cap U_\beta}$ and $\varepsilon = \sum_\alpha i_{U,U_\alpha}$. 
\end{definition}

There is a useful characterization of precosheaves to be cosheaves. 
\begin{proposition}[\cite{Br}]  \label{prop:characterization}
A precosheaf $\A$ on $X$ is cosheaf if and only if it satisfies the following two conditions. 
\begin{itemize}
\item[(a)] For any open sets $U,V \in \mathsf O(X)$, the short complex
\[
\A(U \cap V) \xrightarrow{\Psi} \A(U) \oplus \A(V) \xrightarrow{\varepsilon} \A(U \cup V) \to 0
\]
is exact, where $\varepsilon = i_{U \cup V, U} + i_{U \cup V, V}$ and $\Psi = i_{U, U \cap V} - i_{V, U \cap V}$.
\item[(b)] If a family $\{U_\alpha\}_{\alpha \in A}$ of open sets in $X$ is directed, that is, for any $\alpha, \alpha' \in A$, there is $\alpha'' \in A$ such that $U_\alpha \cup U_{\alpha'} \subset U_{\alpha''}$, then the map 
\[
\varinjlim_{\alpha \in A} \A(U_\alpha) \to \A(\bigcup_{\alpha \in A} U_\alpha)
\] 
is an isomorphism. 
\end{itemize}
\end{proposition}

The short sequence \eqref{eq:short} can be extended on the left side as a 
chain complex as follows. 
For $k \ge 1$, let us define a map 
\[
\Phi_k = \Phi_k^{\mathcal U, \A} : \bigoplus_{\alpha_0, \dots, \alpha_k \in A} \A(U_{\alpha_0 \dots \alpha_k}) \to \bigoplus_{\alpha_0, \dots, \alpha_{k-1} \in A} \A(U_{\alpha_0 \dots \alpha_{k-1}})
\]
associated to a family $\mathcal U = \{U_\alpha\}_{\alpha \in A}$ of open sets, by, 
\[
\Phi_k = \sum_{\alpha_0, \dots, \alpha_k} \sum_{p=0}^k (-1)^p\, i_{\hat U_{\alpha_p}, U_{\alpha_0 \dots \alpha_k}} 
\]
where $U_{\alpha_0 \dots \alpha_k}$ denotes the intersection $\bigcap_{p=0}^k U_{\alpha_p}$ and $\hat U_{\alpha_p}$ is $\bigcap_{j \neq p} U_{\alpha_j}$.
Here, $\Phi_1^{\mathcal U, \A}$ is nothing but $\Phi$ in \eqref{eq:short}.
We set 
\[
\check C_k(\mathcal U, \A) := \bigoplus_{\alpha_0, \dots, \alpha_k \in A} \A(U_{\alpha_0 \dots \alpha_k}). 
\]
The maps $(\Phi_k^{\mathcal U, \A})_{k \ge 1}$ satisfies $\Phi_{k-1}^{\mathcal U, \A} \circ \Phi_k^{\mathcal U, \A} = 0$ for $k \ge 2$. 
So, the group 
\[
\check C_\bullet(\mathcal U, \A) := \bigoplus_{k \ge 0} \check C_k(\mathcal U,\A)
\]
becomes a chain complex with boundary map $\Phi = \Phi^{\mathcal U, \A} = (\Phi_k^{\mathcal U,\A})_{k \ge 1}$ and $\varepsilon : \check C_0(\mathcal U, \A) \to \A(\bigcup_\alpha U_\alpha)$ is an augmentation of this complex.

\begin{proposition}[\cite{Br}] \label{prop:flabby}
Let $\A$ be a flabby cosheaf on $X$. 
Then the sequence 
\[
\cdots \to \check C_k(\mathcal U, \A) 
\xrightarrow{\Phi_k} \check C_{k-1}(\mathcal U, \A) \to 
\cdots \xrightarrow{\Phi_1} 
\check C_0(\mathcal U, \A) \xrightarrow{\varepsilon} \A(X) \to 0
\]
is exact, for any open covering $\mathcal U$ of $X$.
\end{proposition}

For two open coverings $\mathcal U = \{U_i\}_{i \in I}$ and $\mathcal V = \{V_j\}_{j \in J}$ of $X$, we say that $\mathcal V$ is a {\it refinement} of $\mathcal U$ if there is a map $\lambda : J \to I$ between index sets such that $V_j \subset U_{\lambda(j)}$ holds for every $j \in J$. 
Such a map $\lambda$ is called a {\it refinement projection} from $\mathcal V$ to $\mathcal U$, and is denoted by $\lambda : \mathcal V \prec \mathcal U$. 
The refinement projection $\lambda$ induces a map $\lambda_\# : \check C_k(\mathcal V, \A) \to \check C_k(\mathcal U,\A)$ defined by  
\[
\lambda_\# : \A(V_{j_0 \dots j_k}) \to \A(U_{\lambda(j_0)\dots \lambda(j_k)})
\]
for each component of $\check C_k(\mathcal V, \A)$.
Actually, we have 
\begin{lemma} \label{lem:check0}
$\lambda_\# : \check C_\bullet(\mathcal V,\A) \to \check C_\bullet(\mathcal U,\A)$ is a chain map.
\end{lemma}
\begin{proof}
Let us consider the following diagram 
\[
\begin{CD}
&\A(V_{j_0 \dots j_k}) @> \lambda_\# >> &\A(U_{\lambda(j_0) \dots \lambda(j_k)}) \\
&@V i_{\hat V_{j_p}, V_{j_0 \dots j_k}} VV &@VV i_{\hat U_{\lambda(j_p)}, U_{\lambda(j_0) \dots \lambda(j_k)}} V \\
&\A(\hat V_{j_p}) @> \lambda_\# >> &\A(\hat U_{\lambda(j_p)}).
\end{CD}
\]
This diagram commutes from the definition. 
It implies the conclusion of the lemma. 
\end{proof}


\subsection{Double complices}

If $\A$ is a precosheaf of $\cab$-valued, then we denote it by $\A = \A_\bullet$. 
In this case, $\check C_\bullet(\mathcal U, \A_\bullet)$ becomes a double complex by the boundary maps $\Phi$ and $\partial$. 
Indeed, the following holds.
\begin{lemma} \label{lem:check}
We have $\pa \Phi = \Phi \pa$, where $\pa$ acts on each companent of $\check C_p(\U, \A_\bullet) = \bigoplus_{i_0, \dots, i_p} \A_\bullet(U_{i_0 \dots i_p})$ for all $p$.

Further, if $\lambda : \mathcal V \prec \mathcal U$ is a refinement projection, 
then the map $\lambda_\# : \check C_\bullet(\mathcal V, \A_\bullet) \to \check C_\bullet(\U, \A_\bullet)$ is a chain map of double complices. 
\end{lemma}
\begin{proof}
By the definition, the following diagram 
\[
\begin{CD}
& \A_k(U_{i_0 \dots i_p}) @> i_{\hat U_{i_j}, U_{i_0 \dots i_p}} >> & \A_k(\hat U_{i_j}) \\
& @V \partial VV &@VV \partial V \\
& \A_{k-1}(U_{i_0 \dots i_p}) @> i_{\hat U_{i_j}, U_{i_0 \dots i_p}} >> & \A_{k-1}(\hat U_{i_j}) 
\end{CD}
\]
commutes. 
Here, $\hat U_{i_j} = \bigcap_{\ell \neq j} U_{i_\ell}$.
The first statement follows from this and the definition of $\Phi$. 

The second statement follows from Lemma \ref{lem:check0} and a similar argument done there. 
\end{proof}

Now, let us consider an abstract double complex $A = A_{\bullet, \bullet} = \bigoplus_{i,j \ge 0} A_{i,j}$ of nonnegative degrees. 
We denote its boundary maps by 
\[
\Phi = \Phi_i = \Phi_{i,j} : A_{i,j} \to A_{i-1,j} \text{ and } 
\pa = \pa_j = \pa_{i,j} : A_{i,j} \to A_{i,j-1}.
\]
The total complex $\mathcal A_\bullet$ of $A$ is defined by $\mathcal A_m = \bigoplus_{i+j=m} A_{i,j}$ together with the boundary map $\sum_{i+j=m} \Phi_i + (-1)^j \pa_j : \mathcal A_m \to \mathcal A_{m-1}$.
Let $B  = B_\bullet = \bigoplus_{j \ge 0} B_j$ be a chain complex with the boundary map $\pa^B : B_j \to B_{j-1}$.
Suppose that there is a map $\varepsilon : A_{0,\bullet} \to B_\bullet$ consisting of morphisms $\varepsilon_j : A_{0,j} \to B_j$ such that $\varepsilon_j \pa_{j+1} = \pa_{j+1}^B \varepsilon_{j+1}$ and $\varepsilon_j \Phi_{1,j} = 0$ for every $j \ge 0$.
This situation is presented as in the following diagram 
\begin{equation*} \label{eq:diagram1}
{\footnotesize
\hspace{-20pt}
\begin{CD}
@. \vdots @. \vdots @. \vdots @. \vdots \\
@. @V \pa VV @V \pa VV @V \pa VV @V \partial^B VV \\
\cdots @> \Phi >> A_{2,2} @> \Phi >> A_{1,2} @> \Phi >> A_{0,2} @> \varepsilon >> B_2  \\
@. @V \pa VV @V \pa VV @V \pa VV @V \partial^B VV \\
\cdots @> \Phi >> A_{2,1} @> \Phi >> A_{1,1} @> \Phi >> A_{0,1} @> \varepsilon >> B_1  \\
@. @V \pa VV @V \pa VV @V \pa VV @V \partial^B VV \\
\cdots @> \Phi >> A_{2,0} @> \Phi >> A_{1,0} @> \Phi >> A_{0,0} @> \varepsilon >> B_0. \end{CD}
}
\end{equation*}

In this case, a morphism $\varepsilon_\ast : H_m(\mathcal A_\bullet) \to H_m(B_\bullet)$ is defined in a canonical way. 
Then, the following is well-known. 
\begin{lemma}\label{lem:double}
Let $A = (A_{i,j},\Phi,\pa),B = (B_j,\pa^B)$ and $\varepsilon$ be as above. 
Let $m \ge 0$. 
If the sequence 
\[
A_{m-j,j} \xrightarrow{\Phi} A_{m-j-1,j+1} \xrightarrow{\Phi} \cdots \xrightarrow{\Phi} A_{0,j} \xrightarrow{\varepsilon} B_j \to 0
\]
is exact, for each $0 \le j \le m$, then the map $\varepsilon_\ast : H_m(\mathcal A_\bullet) \to H_m(B_\bullet)$ is surjective.
Namely, for any $c \in B_m$ with $\partial^B c=0$, there exist elements $c_{m-k,k} \in A_{m-k,k}$ for $0 \le k \le m$ satisfying 
\[
\varepsilon c_{0,m} = c \text{ and } \Phi c_{m-k, k} = \partial c_{m-k-1, k+1}
\]
for every $k$ with $0 \le k \le m-1$.
\end{lemma}

\subsection{Local triviality of precosheaves}
Let $X$ be a space. 
A precosheaf $\H$ on $X$ is said to be {\it locally trivial} (or be locally zero) if for any $x \in X$ and $U \in \mathsf O(X)$ with $x \in U$, there is $V \in \mathsf O(X)$ with $x \in V \subset U$ such that the map $i_{V,U}^{\,\H} : \H(V) \to \H(U)$ is trivial.

A topological space is said to be paracompact if any open covering of it admits a locally finite refinement.
Recall that any metric space is paracompact.

\begin{theorem}[\cite{Br}] \label{thm:LT}
Let $X$ be a paracompact topological space. 
If $\H$ is a locally trivial precosheaf on $X$, then for any open covering $\mathcal U = \{U_i\}$ of $X$, there is an open refinement $\mathcal V = \{V_j\}$ of $\mathcal U$ with a refinement projection $\lambda : \mathcal V \prec \mathcal U$ such that the map 
\[
\H(V_{j_0 \dots j_m}) \to \H(U_{\lambda(j_0) \dots \lambda(j_m)})
\]
is trivial for every indices $j_0, \dots, j_m$ of $\mathcal V$.
\end{theorem}

\subsection{Local triviality of spaces}
Let us denote by $\mathsf{Met}$ the category of metric spaces and locally Lipschitz maps and by $\mathsf{Top}$ the category of topological spaces and continuous maps. 
Let $\mathsf D$ be one of $\mathsf{Met}$ and $\mathsf{Top}$, and $\mathsf C$ denote one of $\Ab$ and $\cab$.
Let us consider a covariant functor $H : \mathsf D \to \mathsf C$.
Then, for each $X \in \mathsf D$, we obtain a precosheaf $H : \mathsf O(X) \to \mathsf C$ on $X$.
We say that a space $X$ is $H$-{\it locally trivial} if the precosheaf $H$ on $X$ is locally trivial.
In this terminologies, the following holds. 

\begin{proposition} \label{prop:H-LT is stable}
Let $H : \mathsf D \to \mathsf C$ be a covariant functor. 
The $H$-local triviality is inherited to open subsets. 

Further, when $\mathsf D = \met$, the $H$-local triviality 
is stable under locally bi-Lipschitz homeomorphisms.
\end{proposition}
\begin{proof}
Let $X \in \mathsf D$ and $X'$ an open subset of $X$. 
We assume that $X$ is $H$-locally trivial. 
Let us take $x \in X'$ and an open neighborhood $U$ of $x$ in $X'$. 
Since $U$ is open in $X$ and $X$ is $H$-locally trivial, there is $V \in \mathsf O(X)$ with $x \in V \subset U$ such that $H(\iota)=0$, where $\iota : V \to U$ is the inclusion. 
Since $V \in \mathsf O(X')$, we conclude that $X'$ is $H$-locally trivial. 

Further, we assume that $X$ is a metric space and take another metric space $Y$. 
Let $f : X \to Y$ be a locally bi-Lipschitz homeomorphism.
To prove the statement, we may assume that $f$ is a bi-Lipschitz homeomorphism.
Let us take $y \in Y$ and an open neighborhood $V$ of $y$ in $Y$.
Set $x = f^{-1}(y) \in X$, $L = \max \{\Lip(f), \Lip(f^{-1})\}$. 
By the $H$-local triviality of $X$, we obtain $r > 0$ such that $U(x,r) \subset f^{-1}(V)$ and $H(\iota) = 0$, where $\iota : U(x,r) \to f^{-1}(V)$ is the inclusion.
Then, $U(y, L^{-1}r) \subset V$ and the inclusion $\iota' : U(y, L^{-1}r) \to V$ is decomposed as $\iota' = f^{-1} \circ \iota \circ f$. 
Hence, we obtain $H(\iota') = 0$.
This completes the proof.
\end{proof}

\subsection{A way to compare homologies by using cosheaves}

In this subsection, we prove the following important 
\begin{theorem} \label{thm:cosheaf}
Let $\A_\bullet$ and $\A_\bullet'$ be flabby cosheaves on a paracompact topological space $X$ of $\cab$-valued together with a natural transformation $\eta : \A_\bullet \to \A_\bullet'$. 
Let $A$ and $A'$ be precosheaves on $X$ of $\Ab$-valued together with natural transformations $\zeta : A \to A'$, $\xi : \A_0 \to A$ and $\xi' : \A_0' \to A'$ such that $\xi' \eta = \zeta \xi$. 
Suppose that $\xi : \A_\bullet \to A$ and $\xi' : \A_\bullet' \to A'$ are augmentations. 
Further, we assume that $\eta : \A_0(U) \to \A_0'(U)$ is surjective and $\zeta : A(U) \to A'(U)$ is injective for each $U \in \mathsf O(X)$. 
Then, the following holds. 
\begin{itemize}
\item[(1)] If the precosheaves $\tilde H_p(\A_\bullet)$ are locally trivial on $X$ for $0 \le p \le m$, then $\eta_\ast : H_m(\A_\bullet(U)) \to H_m(\A_\bullet'(U))$ is injective for every $U \in \mathsf O(X)$. 
Here, $\tilde H_p(\A_\bullet)$ is the $p$-th homology of the augmented complex $\xi : \A_\bullet \to A$. 
\item[(2)] If the precosheaves 
$\tilde H_p(\A_\bullet)$ and $H_q(\A_\bullet')$ are locally trivial on $X$ for $0 \le p \le m-1$ and 
$1 \le q \le m$, then $\eta_\ast : H_m(\A_\bullet(U)) \to H_m(\A_\bullet'(U))$ is surjective for every $U \in \mathsf O(X)$. 
\end{itemize}
\end{theorem}
Note that the surjectivity of $\eta$ implies that $H_0(\A_\bullet(U)) \to H_0(\A_\bullet'(U))$ is always surjective, for every open set $U \subset X$. 
\begin{proof}
Let $X,\A_\bullet,\A_\bullet',A,A',\eta,\zeta,\xi,\xi'$ be in the assumption. 
First, we assume that the precosheaves $\tilde H_p(\A_\bullet)$ are locally trivial on $X$ for $0 \le p \le m$. 
We prove that $\eta_\ast : H_m(\A_\bullet(X)) \to H_m(\A_\bullet'(X))$ is injective. 
By Theorem \ref{thm:LT}, we obtain a sequence of open coverings $\{\mathcal U_p\}_{p=0}^m$ of $X$ such that $\U_p$ is a refinement of $\U_{p+1}$ together with a refinement projection $\lambda_p : \U_p \prec \U_{p+1}$ and that 
\[
(\lambda_p)_\ast : \tilde H_p(\A_\bullet(U_{i_0\dots i_\ell})) \to \tilde H_p(\A_\bullet(\lambda_p(U_{i_0 \dots i_\ell})))
\]
are trivial maps for all finite elements $U_{i_0}, \dots, U_{i_\ell} \in \U_p$, for each $p = 0, \dots, m-1$. 
Here, we write $\lambda_p(U_{i_0 \dots i_\ell}) = \cap_{a=0}^\ell \lambda_p (U_{i_a})$.
For $\U = \U_p$, we consider the following diagram. 
\begin{equation*} \label{eq:diagram}
{\footnotesize
\hspace{-20pt}
\begin{CD}
@. \vdots @. \vdots @. \vdots @. \vdots \\
@. @V \pa VV @V \pa VV @V \pa VV @V \partial VV \\
\cdots @> \Phi >> \check C_2(\U,\A_2) @> \Phi >> \check C_1(\U,\A_2) @> \Phi >> \check C_0(\U,\A_2) @> \varepsilon >> \A_2(X) @>>> 0\\
@. @V \pa VV @V \pa VV @V \pa VV @V \partial VV \\
\cdots @> \Phi >> \check C_2(\U,\A_1) @> \Phi >> \check C_1(\U,\A_1) @> \Phi >> \check C_0(\U,\A_1) @> \varepsilon >> \A_1(X) @>>> 0\\
@. @V \pa VV @V \pa VV @V \pa VV @V \partial VV \\
\cdots @> \Phi >> \check C_2(\U,\A_0) @> \Phi >> \check C_1(\U,\A_0) @> \Phi >> \check C_0(\U,\A_0) @> \varepsilon >> \A_0(X) @>>> 0\\
@. @V \xi VV @V \xi VV @V \xi VV @. \\
\cdots @. \check C_2(\U,A) @. \check C_1(\U,A) @. \check C_0(\U,A) @. @. 
\end{CD}
}
\end{equation*}
Here, $\partial$ denotes the boundary map of the complex $\A_\bullet$.
Let us take $c \in \A_m(X)$ with $\pa c = 0$. 
Since $\A_\bullet$ is a flabby cosheaf, by Lemma \ref{lem:double}, there are elements $c_{p,m-p} \in \check C_p(\U_0,\A_{m-p})$ such that 
\[
\varepsilon c_{0,m} = c \text{ and }
\pa c_{p,m-p} = \Phi c_{p+1,m-p-1}
\] 
for every $p$ with $1 \le p \le m-1$.
Further, we suppose that $c$ satisfies $\pa' c' = \eta c$ for some $c' \in \A_{m+1}'(X)$, where $\pa'$ is the boundary map of $\A_\bullet'$. 
Then, to prove that $\eta_\ast$ is injective, it suffices to show that there is an element $\bar c \in \A_{m+1}(X)$ such that $\partial \bar c = c$.
Since $\varepsilon : \check C_0(\U_0, \A_{m+1}') \to \A_{m+1}'(X)$ is surjective, there exists $c_{0,m+1}' \in \check C_0(\U_0, \A_{m+1}')$ such that $\varepsilon c_{0,m+1}' = c'$.
Then, we have 
\[
\varepsilon (\partial' c_{0,m+1}' - \eta c_{0,m}) = 0. 
\]
Therefore, there is an element $c_{1,m}' \in \check C_1(\mathcal U_0, \A_m')$ such that 
\[
\Phi' c_{1,m}' = \partial' c_{0,m+1}' - \eta c_{0,m}.
\]
So, we have 
\[
\Phi' \partial' c_{1,m}' = - \eta \partial c_{0,m} = -\eta \Phi c_{1,m-1} = - \Phi' \eta c_{1,m-1}. 
\]
Hence, there is an element $c_{2,m-1}' \in \check C_2(\mathcal U_0, \A_{m-1}')$ such that $\Phi' c_{2,m-1}' = \partial' c_{1,m}' + \eta c_{1,m-1}$. 
Repeating such a diagram chasing, we obtain elements $c_{p+1,m-p}' \in \check C_{p+1}(\mathcal U_0, \A_{m-p}')$ such that 
\begin{align*}
&\Phi' c_{p+1.m-p}' = \partial' c_{p,m-p+1}' + (-1)^{p+1} \eta c_{p, m-p}
\end{align*}
for $1 \le p \le m$.
Since $\eta : \check C_{m+1}(\mathcal U_0,\A_0) \to \check C_{m+1}(\mathcal U_0,\A_0')$ is surjective, there is a $c_{m+1,0} \in \check C_{m+1}(\mathcal U_0,\A_0)$ such that $\eta c_{m+1,0} = c_{m+1,0}'$.
Then, we have 
\[
\zeta \xi \Phi c_{m+1,0} = (-1)^{m+1} \zeta \xi c_{m,0}.
\]
Since $\zeta$ is injective, we obtain 
\[
\xi \Phi c_{m+1,0} = (-1)^{m+1} \xi c_{m,0}. 
\]
By the property of $\lambda_0 : \mathcal U_0 \prec \mathcal U_1$, there is a $c_{m,1} \in \check C_m(\mathcal U_1, \A_1)$ such that 
\[
\partial c_{m,1} = (\lambda_0)_\# (\Phi c_{m+1,0} + (-1)^m c_{m,0}). 
\]
Therefore, we obtain 
\[
\partial \Phi c_{m,1} = (\lambda_0)_\# (-1)^m \partial c_{m-1,1}.
\]
Hence, by the property of $\lambda_1$, we obtain $c_{m-1,2} \in \check C_{m-1}(\mathcal U_2, \A_2)$ such that 
\[
\partial c_{m-1,2} = (\lambda_1)_\# (\Phi c_{m,1} + (-1)^{m+1} (\lambda_0)_\# c_{m-1,1})
\]
Repeating this argument, we have elements $c_{m-p,p+1} \in \check C_{m-p}(\mathcal U_{p+1}, \A_{p+1})$ satisfying 
\[
\partial c_{m-p,p+1} = (\lambda_p)_\# \Phi c_{m-p+1,p} + (-1)^{m+p} (\tilde \lambda_p)_\# c_{m-p,p}
\]
for $1 \le p \le m$, where $\tilde \lambda_p {}_\# = \lambda_p {}_\# \circ \cdots \circ \lambda_0 {}_\#$.
Then, setting $\bar c = \varepsilon c_{0,m+1} \in \A_{m+1}(X)$, we have 
\[
\partial \bar c = \varepsilon c_{0,m} = c.
\]
This implies that $\eta_\ast : \tilde H_m(\A_\bullet(X)) \to \tilde H_m(\A_\bullet'(X))$ is injective. 
That is, this completes the proof of (1). 

Next, we prove that $\eta_\ast : H_m(\A_\bullet(X)) \to H_m(\A_\bullet'(X))$ is surjective, assuming the assumption of (2). 
By a thing mentioned at before starting the proof, we may assume that $m \ge 1$. 
Let us take $c' \in \A_m'(X)$ with $\partial' c' = 0$.
Let $\mathcal W$ be an arbitrary open covering of $X$.
Then, there are $c_{p,m-p}' \in \check C_p(\mathcal W,\A_{m-p}')$ such that 
\begin{equation} \label{eq:c'}
\left\{
\begin{aligned}
\varepsilon c_{0,m}' &= c' \\
\partial' c_{p,m-p}' &= \Phi' c_{p+1,m-p-1}'
\end{aligned}
\right.
\end{equation}
for all $1 \le p \le m-1$.
Since $\eta : \A_0 \to \A_0'$ is surjective, there is a $c_{m,0} \in \check{C}_m(\mathcal W, \A_0)$ such that 
\begin{equation} \label{eq:surj}
\eta c_{m,0} = c_{m,0}'.
\end{equation}
By the assumption, we have 
\[
\zeta \xi \Phi c_{m,0} = \xi' \eta \Phi c_{m,0} = \xi' \Phi' c_{m,0}' = \xi' \partial' c_{m-1,1}' = 0. 
\]
Since $\zeta : A \to A'$ is injective, we conclude 
\begin{equation} \label{eq:inj}
\xi \Phi c_{m,0}= 0.
\end{equation} 

Now, let us consider the following sequence of open coverings of $X$ such that 
\[
\mathcal V_0 \prec \U_1 \prec \mathcal V_1 \prec \cdots \prec \mathcal V_{k-1} \prec \mathcal U_{k} \prec \mathcal V_{k} \prec \cdots \prec \mathcal V_{m-1} \prec \U_m.
\]
Here, the refinemet projections are denoted by 
\[
\lambda_{k-1} : \mathcal V_{k-1} \prec \U_k \text{ and } \mu_\ell : \mathcal U_\ell \prec \mathcal V_\ell
\]
for $1 \le k \le m$ and $1 \le \ell \le m-1$. 
By Theorem \ref{thm:LT}, we may assume that the induced maps 
\begin{equation} \label{eq:LT}
\begin{aligned}
(\lambda_{k-1})_\ast &: \tilde H_{k-1}(\A_\bullet(V_{j_0 \dots j_p})) \to \tilde H_{k-1}(\A_\bullet(\lambda_{k-1}(V_{j_0 \dots j_p}))) \\
(\mu_\ell)_\ast &: H_\ell(\A_\bullet'(U_{i_0 \dots i_p})) \to H_\ell(\A_\bullet'(\mu_k(U_{i_0 \dots i_p}))
\end{aligned}
\end{equation}
are trivial, for all $p \ge 0$, $V_{j_0}, \dots, V_{j_p} \in \mathcal V_{k-1}$, $U_{i_0}, \dots, U_{i_p} \in \U_k$, $1 \le k \le m$ and $1 \le \ell \le m-1$.
As seen above, we have a sequence $\{ c_{p,m-p}' \in \check C_p(\mathcal V_0, \A_{m-p}') \}_{0 \le p \le m}$ and an element $c_{m,0} \in \check C_m(\mathcal V_0,\A_m)$ satisfying \eqref{eq:c'}, \eqref{eq:surj} and \eqref{eq:inj}. 
By \eqref{eq:inj} and the triviality of \eqref{eq:LT}, there is an element $c_{m-1,1} \in \check C_{m-1}(\U_1, \A_1)$ such that 
\[
\partial c_{m-1,1} = \lambda_0{}_\# \Phi c_{m,0}. 
\]
Then, we have 
\begin{align*}
\partial \Phi c_{m-1,1} &= \lambda_0{}_\# \Phi^2 c_{m,0} = 0, \\
\partial' \eta c_{m-1,1}&= \lambda_0{}_\# \Phi' c_{m,0}' = \partial' \lambda_0{}_\# c_{m-1,1}'.
\end{align*}
By the second equality and the triviality of $(\mu_1)_\ast$, there is an element $c_{m-1,2}' \in \check C_{m-1}(\mathcal V_1, \A_2')$ such that 
\[
\partial' c_{m-1,2}' = \mu_1{}_\# (\eta c_{m-1,1} - \lambda_0{}_\# c_{m-1,1}').
\]
The rest equality and the triviality of $(\lambda_1)_\ast$ guarantee the existence of an element $c_{m-2,2} \in \check C_{m-2}(\mathcal U_2, \A_2)$ satisfying 
\[
\partial c_{m-2,2} = \lambda_1{}_\# \mu_1{}_\# \Phi c_{m-1,1}.
\]
Repeating such an argument, we obtain sequences of elements $c_{m-p,p+1}' \in \check C_{m-p}(\mathcal V_p, \A_{p+1})$ and $c_{m-p,p} \in \check C_{m-p}(\U_p, \A_p)$ such that 
\begin{align*}
\partial' c_{m-p,p+1}' &= \mu_p{}_\# (\eta c_{m-p,p}-\nu_{p-1} c_{m-p,p}'), \\
\partial c_{m-p,p} &= \lambda_{p-1}{}_\# \mu_{p-1} {}_\# \Phi c_{m-p+1,p-1}
\end{align*}
for all $2 \le p \le m$, where $\nu_{p-1} = \lambda_{p-1}{}_\# \mu_{p-1}{}_\# \cdots \lambda_1{}_\# \mu_1 {}_\# \lambda_0 {}_\#$. 
Let us consider elements $c := \varepsilon c_{0,m} \in \A_m(X)$ and $\bar c' := \varepsilon c_{0,m+1}' \in \A_{m+1}'(X)$. 
By the construction, they satisfy 
\[
\eta c = c' + \partial' \bar c'.
\]
Therefore, we know that $\eta_\ast : H_m(\A_\bullet(X)) \to H_m(\A_\bullet'(X))$ is surjective. 
This completes the proof of Theorem \ref{thm:cosheaf}. 
\end{proof}

\subsection{Spaces of currents as cosheves}
Let $X$ be a metric space.
For each open set $U \in \mathsf O(X)$, we obtain chain complicies $\N_\bullet^\cpt(U)$ and $\I_\bullet^\cpt(U)$. 
Assignments $U \mapsto \N_\bullet^\cpt(U)$ and $U \mapsto \I_\bullet^\cpt(U)$ are precosheaves on $X$ of $\cab$-valued, by the definition. 
We have 
\begin{lemma} \label{lem:I is cosheaf}
The precosheaves $\N_\bullet^\cpt$ and $\I_\bullet^\cpt$ on $X$ are actually flabby cosheaves. 
\end{lemma}
\begin{proof}
Let us denote by $\mathbf C_k$ one of $\N_k^\cpt$ and $\I_k^\cpt$. 
By Lemma \ref{lem:coflabby}, the precoseaf $\mathbf C_k$ is flabby. 
By Lemma \ref{lem:a}, we already know that $\mathbf C_\bullet$ satisfies the condition (a) in Proposition \ref{prop:characterization}. 
We prove the condition (b) in Proposition \ref{prop:characterization}.
Let $\{U_\alpha\}$ be a directed family of open sets. 
Since $\mathbf C_k$ is flabby, the map $\mathbf C_k(U_\alpha) \to \mathbf C_k(U)$ is injective for every $\alpha$, where $U = \bigcup_\alpha U_\alpha$. 
Because the functor taking the direct limit is exact, the canonical map 
\[
\varinjlim \mathbf C_k(U_\alpha) \to \mathbf C_k(U)
\]
is injective. 
Let us prove that this map is surjective. 
Let $T \in \mathbf C_k(U)$. 
Since $T$ has a compact support and $\{U_\alpha \}$ is directed, there is an $\alpha$ such that $\spt(T) \subset U_\alpha$. 
By Lemma \ref{lem:cpt spt}, $T$ can be regarded as a current in $U_\alpha$. 
This implies that the considered map is surjective. 
This completes the proof. 
\end{proof}

\subsection{Singular (Lipschitz) coheaves} \label{subsec:singular cosheaf}
For a topological space $X$, the singular chain complex $S_\bullet$ of each open set gives a flabby precoheaf on $X$. 
In general, $S_\bullet$ is not a cosheaf. 
Further, taking subdivisions infinitely many times, we obtain a coheaf on $X$ as follows. 
\begin{example}[\cite{Br}] \upshape \label{ex:singular cosheaf}
For each topological space $X$, 
let us consider a sequence of barycentric subdivisions 
\[
S_\bullet(X) \xrightarrow{\mathrm{Sd}} S_\bullet(X) \xrightarrow{\mathrm{Sd}} \cdots \xrightarrow{\mathrm{Sd}} S_\bullet(X) \xrightarrow{\mathrm{Sd}} \cdots
\]
and its direct limit, denoted by $\mathfrak S_\bullet(X)$. 
In this case, for degree $k$, the direct limit $\mathfrak S_k(X)$ is represented as the quotient group of $S_k(X)$ identifying $c$ and $c'$ whenever $\mathrm{Sd}^m c = \mathrm{Sd}^{m'} c'$ for some $m , m' \ge 0$. 

Then, a correspondence $\mathsf O(X) \ni U \mapsto \mathfrak S_\bullet(U) \in \cab$ is a flabby cosheaf on $X$. 
Indeed, because $S_\bullet$ is flabby and the direct limit $\varinjlim$ is an exact functor, $\mathfrak S_\bullet$ is flabby. 
To prove that $\mathfrak S_\bullet$ is a cosheaf, it suffices to check the properties (a) and (b) in Proposition \ref{prop:characterization}. 
However, it is trivial, by the definition. 
Finally, noticing that the identity map on $S_\bullet$ and the subdivision are chain homotopy equivalent, we have a natural isomorphism 
\[
\eta_X : H_\ast(X) \cong H_\ast(\mathfrak S_\bullet(X)). 
\]
Here, the naturality means the functorial sense, that is, for a continuous map $f : X \to Y$ between topological spaces, the induced maps $f_\ast : H_\ast(X) \to H_\ast(Y)$ and $f_\ast : H_\ast(\mathfrak S_\bullet(X)) \to H_\ast(\mathfrak S_\bullet(Y))$ satisfy $\eta_Y f_\ast = f_\ast \eta_X$.
\end{example}

In Example \ref{ex:singular cosheaf}, instead of a topological space and singular chains with a metric space and singular Lipschitz chains, respectively, we obtain a flabby cosheaf $\mathfrak S_\bullet^\Lip$ on a metric space $X$ of $\cab$-valued. 
That is, we set 
\[
\mathfrak S_\bullet^\Lip (U) = \varinjlim (S_\bullet^\Lip(U) \xrightarrow{\mathrm{Sd}} S_\bullet^\Lip(U) \xrightarrow{\mathrm{Sd}} \cdots )
\]
for each $U \in \mathsf O(X)$. 
Here, we note that the subdivision preserves the Lipschitz-ness of singular chains. 
So, the map $\mathrm{Sd} : S_\bullet^\Lip \to S_\bullet^\Lip$ is well-defined. 
Further, since the canonical chain homotopy equivalence maps between the identity and the subdivision on $S_\bullet$ also preserve the Lipschitz-ness of singular chains, they give chain homotopy equivalence between $S_\bullet^\Lip$ and $\mathfrak S_\bullet^\Lip$. 
Therefore, we have a natural isomorphism 
\[
H_\ast^\Lip(X) \cong H_\ast(\mathfrak S_\bullet^\Lip(X))
\]
between their homologies. 

Let us recall that the natural map $[\,\cdot\,] = [\,\cdot\,]_X : S_\bullet^\Lip(X) \to \I_\bullet^\cpt(X)$ for each metric space $X$ was defined in Section \ref{sec:current}.
Obviously, we have $[\mathrm{Sd}\, c] = [c]$ for singular Lipschitz chain $c$. 
Hence, we can define a natural map 
\[
[\,\cdot\,] : \mathfrak S_\bullet^\Lip(X) \to \I_\bullet^\cpt(X). 
\]

\subsection{Local Lipschitz contractibility implying local triviality}
We prove 
\begin{lemma} \label{lem:LLC to LT}
If a metric space $X$ is locally Lipschitz contractible, then it is $H$-locally trivial. 
Here, $H$ is one of the procosheaves $H_k^\sing$, $\tilde H_k^\Lip$, $H_k^\Lip$ and $H_k^\IC$ for $k \ge 0$. 
\end{lemma}
\begin{proof}
Since an LLC metric space $X$ is locally contractible in the usual sense, 
the statement for $H = H_\ast^\sing$ holds. 
For another $H$, the statement follows fom Lemmas \ref{lem:homotopy IC}, \ref{lem:dim axiom} and \ref{lem:homotopy L}. 
\end{proof}

\subsection{Cone inequalities implying $H$-locally triviality}
Riedweg and Sch\"appi introduced the notion of metric spaces satisfying the cone inequalities. 
We translate this notion in terms of precosheaves. 

Let $X$ be a metric space and $\mathfrak C_\bullet = (\mathfrak C_\bullet, \partial)$ a flabby 
precosheaf on $X$ of $\cab$-valued. 
Further, $A$ is an $\Ab$-valued precosheaf on $X$ together with a natural transformation $\varepsilon : \mathfrak C_0 \to A$ such that $\varepsilon
 \partial = 0$. 
That is, $\varepsilon : \mathfrak C_\bullet \to A$ is an augmentation. 
The map $\varepsilon$ may be trivial. 
Moreover, we suppose the following. 
\begin{itemize}
\item[(1)] For any $j \ge 0$, 
$c \in \mathfrak C_j(X)$, 
there is a unique compact set $K(c)$ 
such that for every $V \in \mathsf O(X)$ with $K(c) \subset V$, 
there is an element $c' \in \mathfrak C_j(V)$ such that $i_\# c' = c$. 
Here, $i : V \to U$ is the inclusion and $i_\# = \mathfrak C_\bullet(i)$ denotes the induced map; 
\item[(2)] For any $j \ge 1$, $U \in \mathsf O(X)$, $c \in \mathfrak C_j(U)$, we have $K(\partial c) \subset K(c)$. 
\end{itemize}

For instance, $S_\bullet$ and $S_\bullet^\Lip$ satisfy these conditions for the canonical augmentations.
In these cases, $K(c)$ is the image of a singular chain $c$.
Further, $\N_\bullet^\cpt$ and $\I_\bullet^\cpt$ also satisfy the conditions, for augmentations $\varepsilon : \N_0^\cpt(X) \to \mathbb R$ and $\varepsilon : \I_0^\cpt(X) \to \mathbb Z$ defined by $\varepsilon(T) = T(1)$. 
In these cases, $K(S)$ is the support of a current $S$. 

\begin{definition}[\cite{RS}] \upshape \label{def:cone ineq}
Let $X, \mathfrak C_\bullet, A, \varepsilon$ be as above.  
We say that $X$ admits {\it the cone inequality for $\mathfrak C_j$} if for any $x \in X$, there exist $r > 0$ and a continuous non-decreasing function $F : [0,\infty) \to [0,\infty)$ with $F(0) = 0$ such that 
\begin{itemize}
\item when $j \ge 1$, for every $c \in \mathfrak C_j(X)$ with $K(c) \subset U(x,r)$ and $\partial c = 0$, there is $c' \in \mathfrak C_{j+1}(X)$ such that $\partial c' = c$ with $\mathrm{diam}\, K(c') \le F(\mathrm{diam}\, K(c))$; 
\item when $j = 0$, for every $c \in \mathfrak C_0(X)$ with $K(c) \subset U(x,r)$ and $\varepsilon c = 0$, there is $c' \in \mathfrak C_1(X)$ such that $\partial c' = c$ with $\mathrm{diam}\, K(c') \le F(\mathrm{diam}\, K(c))$.
\end{itemize}
\end{definition}

\begin{lemma} \label{lem:cone to LT}
If $X$ admits the cone inequality for $\mathfrak C_j$, then the precosheaf $\tilde H_j(\mathfrak C_\bullet)$ is locally trivial. 
Here, $\tilde H_j(\mathfrak C_\bullet)$ denotes the $j$-th homology of the augmented complex $\cdots \xrightarrow{\partial} \mathfrak C_j \xrightarrow{\partial} \mathfrak C_{j-1} \xrightarrow{\partial} \cdots \xrightarrow{\partial} \mathfrak C_0 \xrightarrow{\varepsilon} A$. 
\end{lemma}
\begin{proof}
Let $X$ admit a local cone inequality for $\C_j$. 
Then, for $x \in X$, there exists $r > 0$ and $F : [0,\infty) \to [0,\infty)$ satisfying the condition written above Definition \ref{def:cone ineq}.
For any $s \in (0,r)$, we choose $s' \in (0,s)$ with 
\[
s' + F(2 s') < s.
\]
Let us take any $c \in \C_j(X)$ with $K(c) \subset U(x,s')$. 
Such a $c$ is considered as an element in $\C_j(U(x,s'))$, because $\C_j$ is flabby and it satisfies the condition (1).
We suppose that $\partial c = 0$ when $j \ge 1$ and that $\varepsilon c = 0$ when $j = 0$. 
Since $s' < s < r$, there exists $c' \in \C_{j+1}(X)$ such that $\partial c' = c$ and 
\[
\mathrm{diam}\, \,K(c') \le F(\mathrm{diam}\, K(c)).
\]
For any $y \in K(c) \subset K(c')$ and $z \in K(c')$, we have 
\[
d(x,z) \le d(x,y) + d(y,z) < s' + F(2 s') < s.
\]
Hence, $K(c') \subset U(x,s)$.
So, the $c'$ can be regarded as an element in $\C_{j+1}(U(x,s))$. 
Therefore, the morphism
\[
\tilde H_j(\C_\bullet(U(x,s'))) \to \tilde H_j(\C_\bullet(U(x,s)))
\]
is trivial. 
This completes the proof.
\end{proof}

Riedweg and Sch\"appi claimed 
\begin{theorem}[\cite{RS}]\label{thm:RS1}
If a metric space $X$ admits the cone inequalities for $H_k$, $\tilde H_k^\Lip$ and $H_k^\IC$, for all $k$ with $0 \le k \le m$, then the canonical maps 
$H_m(X) \leftarrow H_m^\Lip(X) \to H_m^\IC(X)$ are isomorphisms. 
\end{theorem}
By Lemmas \ref{lem:LLC to LT} and \ref{lem:cone to LT}, our Theorem \ref{thm:cosheaf} is a generalization of Theorem \ref{thm:RS1} in terms of local triviality. 

\subsection{Proof of Theorem \ref{thm:main thm}}
By above preparations, 
we immediately get a proof of our main theorem.  
\begin{proof}[Proof of Theorem \ref{thm:main thm}]
Let $X$ be an LLC metric space. 
By Lemma \ref{lem:LLC to LT}, all the precosheaves $H_k^\sing$, $\tilde H_k^\Lip$ and $H_k^\IC$ on $X$ are locally trivial, for every $k \ge 0$. 
By Lemma \ref{lem:0-chains}, the chain map $[\,\cdot\,] : S_0^\Lip(X) \to \I_0^\cpt(X)$ is an isomorphism. 
Due to Subsection \ref{subsec:singular cosheaf}, the functor $\mathfrak S_\bullet^\Lip : \mathsf O(X) \to \cab$ is a flabby cosheaf and $\mathfrak S_0^\Lip = S_0^\Lip$ by the definition. 
Therefore, Theorem \ref{thm:main thm} follows from those things and Lemmas \ref{lem:augmentation} and \ref{lem:I is cosheaf} and Theorem \ref{thm:cosheaf}. 
\end{proof}

\begin{proof}[Proof of Corollary \ref{cor:main cor}]
Let $H$ denote one of $H_\ast^\Lip$ and $H_\ast^\IC$. 
Let us set $\iota : H \to H_\ast^\sing$ the natural isomorphism obtained in Theorem \ref{thm:main thm}. 
Let $X$ and $Y$ be LLC metric spaces, and $f : X \to Y$ a continuous map. 
Then, we define a homomorphism $H(f) : H(X) \to H(Y)$ by $H(f) = \iota_Y^{-1} \circ H_\ast^\sing(f) \circ \iota_X$. 
Further, for another continuous map $g : Y \to Z$ to an LLC metric space $Z$, we obtain 
\begin{align*}
H(g) \circ H(f) 
&= \iota_Z^{-1} \circ H_\ast^\sing(g) \circ \iota_Y \circ \iota_Y^{-1} \circ H_\ast^\sing(f) \circ \iota_X \\
&= \iota_Z^{-1} \circ H_\ast^\sing(g) \circ H_\ast^\sing(f) \circ \iota_X \\
&= \iota_Z^{-1} \circ H_\ast^\sing(g \circ f) \circ \iota_X \\
&= H(g \circ f).
\end{align*}
This shows that $H$ is extended as a functor on the category of LLC metric spaces and continuous maps such that $H$ is naturally isomorphic to the functor $H_\ast^\sing$. 
Further, if $h : X \times [0,1] \to Y$ is a continuous homotopy, then we have 
\[
H(h_0) = \iota_Y^{-1} \circ H_\ast^\sing(h_0) \circ \iota_X = \iota_Y^{-1} \circ H_\ast^\sing(h_1) \circ \iota_X = H(h_1). 
\]
This implies the homotopy invariance of $H$. 
This completes the proof of Corollary \ref{cor:main cor}. 
\end{proof}

\section{Several remarks} \label{sec:remarks}

\subsection{Relative homologies}
A relative version of the singular Lipschitz homology is defined in a similar way to define the relative singular homology. 
For a subset $A$ in a metric space $X$, the inclusion $i : A \to X$ induces an injective morphism $i_\# : S_\bullet^\Lip(A) \to S_\bullet^\Lip(X)$.  
So, we have a chain complex $S_\bullet^\Lip(X,A)$ as the quotient of $S_\bullet^\Lip(X)$ modulo $S_\bullet^\Lip(A)$.
Its homology $H_\ast(S_\bullet^\Lip(X,A))$ is called the relative singular Lipschitz homology and is denoted by $H_\ast^\Lip(X,A)$. 

Let $(X,A)$ be as above. 
The pushforward $i_\# : \I_\bullet^\cpt(A) \to \I_\bullet^\cpt(X)$ is injective by Lemma \ref{lem:coflabby}. 
Hence, we obtain a chian complex $\I_\bullet^\cpt(X,A) = \I_\bullet(X) / i_\# \I_\bullet^\cpt(A)$. 
Its homology $H_\ast(\I_\bullet^\cpt(X,A))$ is called the relative homology of integral currents with compact support and is denoted by $H_\ast^\IC(X,A)$. 

Let $(Y,B)$ be another pair of metric spaces. 
A map $f$ from $(X,A)$ to $(Y,B)$ is a map $f : X \to Y$ with $f(A) \subset B$.
We say that a map $f : (X,A) \to (Y,B)$ is (locally) Lipschitz, if $f : X \to Y$ is (locally) Lipschitz. 
Obviously, all pairs of metric spaces and all locally Lipscihtz maps give a category. 
We also have 
\begin{theorem} \label{thm:relative}
On the category of all pairs of LLC metric spaces and all locally Lipscihtz maps, the functors $H_\ast^\sing$, $H_\ast^\Lip$ and $H_\ast^\IC$ are natually ismorphic, where $H_\ast^\sing$ denotes the usual relative singular homology. 

In particular, $H_\ast^\Lip$ and $H_\ast^\IC$ can be extended functors on the category of pairs of LLC spaces and continuous maps. 
\end{theorem}
\begin{proof}
This follows from Theorem \ref{thm:main thm} and the five lemma.
\end{proof}

\subsection{Reduced homology of metric currents}
Let $X$ be a metric space. 
We consider a map 
\[
\varepsilon : \I_\bullet^\cpt(X) \ni T \mapsto T(1) \in \mathbb Z. 
\]
This is an augmentation of the complex $\I_\bullet^\cpt(X)$. 
Actually, for $S \in \I_1^\cpt(X)$, we have 
\[
\varepsilon \partial S = S(1,1) = 0.
\]
So, we obtain the reduced homology of $\I_\bullet^\cpt(X)$ augmented by $\varepsilon$, which is denoted by $\tilde H_\ast^\IC(X)$. 
As Theorem \ref{thm:main thm}, we have 
\begin{theorem}
$\tilde H_\ast^\IC$ is actually a functor on the category of LLC metric spaces and locally Lipschitz maps. 
On that category, the functors $\tilde H_\ast^\sing$, $\tilde H_\ast^\Lip$ and $\tilde H_\ast^\IC$ are naturally isomprphic. 
Here, $\tilde H_\ast^\sing$ denotes the usual reduced singular homology. 

In particular, $\tilde H_\ast^\Lip$ and $\tilde H_\ast^\IC$ can be extended to functors on the category of LLC metric spaces and continuous maps. 
\end{theorem}
\begin{proof}
Let $X_0$ be a set of a single point. 
Then, all the homologies $\tilde H_\ast^\sing(X)$, $\tilde H_\ast^\Lip(X)$ and $\tilde H_\ast^\IC(X)$ are represented as the kernels of $\pi_\ast$ between correspondence non-reduced homologies induced by the canoncial map $\pi : X \to X_0$. 
By the naturality of the non-reduced homologies (Theorem \ref{thm:main thm}), we obtain the conclusion of the theorem. 
\end{proof}

\subsection{A remark on the axiom of finite mass}
\begin{remark}\label{rem:tight} \upshape
Our definition (Definition \ref{def:current}) and the original definition given in \cite{AK} of currents are slightly different. 
The main different point is the finite mass axiom. 
Further, in \cite{AK}, it was supposed that all sets satisfy some set-theoretical axiom about the cardinalities. 

The original metric currents were defined only on {\it complete} metric spaces assuming the set-theoretical axiom (\cite{AK}). 
The set-theoretical axiom implies that any fnite Borel measure on every complete metric space is automatically tight.
Therefore, the original definition did not impose that the mass measures of currents are tight. 
On the other hands, the LLC-condition is an open property (Proposition \ref{prop:LLC is preserving}). 
Therefore, if we employ the original definition of metric currents, then an area which is applicable to our theory is very small. 
For instance, if a metric space is complete and LLC, then its open set is LLC, but is not complete, in general. 
Further, we want to ignore an additional set-theoretical axiom. 

Fortunately, as mentioned in \cite{AK}, 
if one deal with only metric currents having tight mass measures, then such currents satisfy all the same results obtained there, further, they can be defined on all metric spaces without the set-theoritical axiom. 
This is the reason why we used currents with tight mass measures. 
\end{remark}

\subsection{Alexandrov spaces revisited}
As mentioned in Subsection \ref{subsec:example}, any finite dimensional Alexandrov space is SLLC. 
The proof of it was based on the theory of gradient flows of distance functions founded by Perelman and Petrunin \cite{PP}, \cite{Pet}. 
Actually, in \cite{MY}, we proved that any point $x$ in an Alexandrov space $X$ has a positive number $r$ such that the distance function $d$ from the metric sphere $S(x,2r)$ centered at $x$ of radius $2 r$ is regular on $U(x,r) \setminus \{x\}$ and further that the absolute gradient $|\nabla d|$ is uniformly bounded on $U(x,r) \setminus \{x\}$. 
Here, $|\nabla d|(y) = \limsup_{z \to y} \frac{|d(z)-d(y)|}{d(z,y)}$. 
Then, the gradient flow of $d$ gives a strong Lipschitz contraction from $U(x,r)$ to $x$. 
On the other hands, extremal subsets of $X$, introduced by Perelman and Petrunin \cite{PP:ext}, have well-behavior in the gradient flows of distance functions. 
Indeed, extremal subsets are characterized by the property that they are preserved under the gradient flow of any distance functions (see \cite{Pet}). 
This fact and the proof of the main result in \cite{MY} implies
\begin{theorem} \label{cor:extremal}
Any extremal subset in an Alexandrov space is strongly locally Lipschitz contractible. 
In particular, the boundary of an Alexandrov space is strongly locally Lipschitz contractible. 
\end{theorem}

Due to Theorem \ref{thm:relative} and Theorem \ref{cor:extremal}, we have 
\begin{corollary} \label{cor:extremal}
Let $X$ be a finite dimensional Alexandrov space and $E$ a its subset. 
Suppose that $E$ belongs to one of three classes of sets in the following: open subsets, discrete subsets, and extremal subsets. 
Then, we have natural isomorphisms $H_\ast(X,E) \cong H_\ast^\Lip(X,E) \cong H_\ast^\IC(X,E)$. 
\end{corollary}

\subsection{An LLC space not having the homotopy type of CW-complices}
This subsection is devoted to prove 
\begin{theorem}[cf.\! \cite{Bo}, \cite{S}] \label{thm:B}
There is an LLC metric space such that it does not have the homotopy type of CW-complices. 
\end{theorem}
Indeed, a space satisfying the topological property written in Theorem \ref{thm:B} was constructed by Borsuk (\cite{Bo}). 
We prove that such a space can admit an LLC metric. 
We also refer Chapter 6 of the book \cite{S} for the construction and recall terminologies used there. 

A metrizable space $X$ is called an {\it ANR} ({\it absolute neighborhood retract}) if it is a neighborhood retract of an arbitrary metrizable space that contains $X$ as a closed subset. 
Here, a closed subset $A$ of a space $Y$ is called a neighborhood retract if there exist a neighborhood $U$ of $A$ and a continuous map $r : U \to A$ such that $r |_A = \mathrm{id}_A$. 
For an open covering $\mathcal V$ of a space $X$, we say that two maps $f,g : Y \to X$ from a space $Y$ are $\mathcal V$-{\it close} if for any $y \in Y$, there is a $V \in \mathcal V$ such that $f(y), g(y) \in V$. 
Let $\U$ be an open covering of $X$ such that $\mathcal V$ is a refinement of $\U$. 
We say that $\mathcal V$ is an $h$-{\it refinement} of $\mathcal U$ if any two $\mathcal V$-close continuous maps $f,g : Y \to X$ from a metrizable space $Y$ are $\U$-homotopic. 
Here, $f$ and $g$ are $\U$-homotopic if there is a continuous map $h : X \times [0,1] \to Y$ such that $h_0=f$, $h_1=g$ and that for any $x \in X$, there is a $U \in \U$ such that $h (\{x\} \times [0,1]) \in U$. 
\begin{lemma} \label{lem:ANR}
Let $X$ be a metrizable space. 
If the open cover $\{X\}$ consisting of only $X$ has no $h$-refinement, then $X$ does not have the homotopy type of absolute neighborhood retracts. 
In particular, $X$ does not have the homotopy type of CW-complicies. 
\end{lemma}
\begin{proof}
We suppose that there is an ANR $Y$ such that $X$ and $Y$ are homotopic. 
Let $\phi : X \to Y$ be a homotpy equivalence. 
By Corollary 6.3.5 in \cite{S}, the cover $\{Y\}$ of $Y$ has an $h$-refinement $\mathcal U$.
Let us set $\mathcal V = \{\phi^{-1}(U) \mid U \in \mathcal U\}$. 
Then, $\mathcal V$ is an $h$-refinement of $\{X\}$, which contradicts to the assumption.  
Indeed, we take $\mathcal V$-close maps $f,g : Z \to X$. 
Then, $\phi \circ f$ and $\phi \circ g$ are $\mathcal U$-close. 
Since $\mathcal U$ is an $h$-refienment of $\{Y\}$, there is a homotopy $h : Z \times [0,1] \to Y$ such that $h_0 = \phi \circ f$ and $h_1 = \phi \circ g$. 
Let $\psi : Y \to X$ be a homotopy inverse of $\phi$. 
By using a homotopy $\psi \circ h$ between $\psi \circ \phi \circ f$ and $\psi \circ \phi \circ g$, we obtain a homotopy between $f$ and $g$. 
This completes the proof of the first statement.
Since every CW-complex is an ANR, the latter statement follows. 
\end{proof}

For a metric space $U$ and its subset $U_1$, we say that $U_1$ is a {\it Lipschitz deformation retract of} $U$ if there exists a Lipschitz homotopy $h : U \times [0,1] \to U$ such that $h_0(x)=x$, $h_1(x) \in U_1$ and $h_1(y) = y$ for every $x \in U$ and $y \in U_1$. 
Such a map $h$ is called a {\it Lipschitz deformation retraction from $U$ to $U_1$}. 
A Lipschitz contraction is a special Lipschitz deformation retraction. 

\begin{lemma} \label{lem:DR}
Let $V$ be a metric space and $V_1$ and $V_0$ its subsets. 
Suppose that $V_0 \subset V_1 \subset V$ and that $V_0$ is a Lipschitz deformation retract of $V_1$ and $V_1$ is a Lipschitz deformation retract of $V$. 
Then, $V_0$ is a Lipschitz deformation retract of $V$. 
\end{lemma}
\begin{proof}
Let us take Lipschitz deformation retractions $h$ from $V$ to $V_1$ and $g$ from $V_1$ to $V_0$. 
Then, we define a map $k : V \times [0,1] \to V$ by 
\[
k(x,t) = \left\{ 
\begin{aligned}
& h(x,2 t) && \text{if } t \le 1/ 2, \\
& g(h_1(x), 2 t -1) && \text{if } t \ge 1/2.
\end{aligned}
\right.
\]
This map is well-defined. 
Further, it is Lipschitz. 
Indeed, for $x, y \in V$ and for $t,s \in [0,1]$, we have 
\begin{align*}
d(k(x,t),k(y,t)) &\le \max\{\Lip(g),1\} \max\{\Lip(h),1\} d(x,y), \\
d(k(x,t),k(x,s)) &\le 2 \max\{\Lip(h), \Lip(g)\} |s-t|. 
\end{align*}
By the construction, $k$ is a deformation retraction from $V$ to $V_0$ in the usual sense. 
Therefore, $V_0$ is a Lipschitz deformation retract of $V$. 
This completes the proof of the lemma.
\end{proof}

\begin{proof}[Proof of Theorem \ref{thm:B}]
Let $\ell_2$ denote the standard Hilbert space of coutably inifinite dimension. 
Let $Q$ be a subset of $\ell_2$ defined by 
\[
Q = \{(x_k)_{k =0}^\infty \in \ell_2 \mid 0 \le x_k \le 2^{-k} \}.
\]
We consider the following spaces. 
\begin{align*}
X_0 &= \{(x_k) \in Q \mid x_0 = 0 \}, \\
C_n &= \{(x_k) \in Q \mid (n+1)^{-1} \le x_0 \le n^{-1}, x_k = 0 \text{ for } k > n \}, \\
X_n &= \pa C_n \text{ (the boundary $n$-sphere of the $(n+1)$-cube $C_n$)},
\end{align*}
where $n \ge 1$. 
Then, we prove that $X = \bigcup_{n=0}^\infty X_n$ is the desired space. 
By the construction, this space is homeomorphic to the space in Theorem 6.3.8 in \cite{S}. 
Therefore, the cover $\{X\}$ has no $h$-refinement. 
By Lemma \ref{lem:ANR}, $X$ does not have the homotopy type of CW-complices.
 
We prove that $X$ is LLC. 
Since an open set $\bigcup_{n=1}^\infty X_n = X \setminus X_0$ in $X$ is a locally finite simplical complex, it is SLLC. 
Let $x \in X_0$. 
We denote by $p_m : X \to \prod_{k=0}^m [0,2^{-k}]$ the projection into the first $(m+1)$-coordinates. 
For any $r > 0$, there exist an $m \ge 1$ and a convex neighborhood $W$ of $(x_1, \dots, x_m)$ in $\prod_{k=1}^m [0,2^{-k}]$ such that $p_m^{-1}([0,m^{-1}] \times W) \subset U(x,r) \cap X$, where $U(x,r)$ is the open ball in $\ell_2$ centered at $x$ of radius $r$.  
Indeed, for $y \in p_m^{-1}([0,m^{-1}] \times W)$, we have 
\[
\|x-y\|_{\ell_2}^2 \le m^{-2} + \sum_{k=1}^m |x_k - y_k|^2 + \sum_{k>m} 4^{-k}.
\]
Hence, we can have such an $m$ and a $W$. 
Then, we define a neighborhood $V$ of $x$ in $X$ by 
\[
V = \left\{
\begin{aligned}
& p_{m+1}^{-1}\left([0,m^{-1}] \times W \times [0,2^{-(m+1)})\right) && \text{if } x_{m+1} \le 4^{-(m+1)}, \\
& p_{m+1}^{-1}\left([0,m^{-1}] \times W \times (0,2^{-(m+1)}]\right) && \text{if } x_{m+1} > 4^{-(m+1)}. 
\end{aligned}
\right. 
\]
Further, we consider the following sets. 
\begin{align*}
V_0 &= \left\{
\begin{aligned}
& p_{m+1}^{-1}(\{0\} \times W \times \{0\}) &&\text{if } x_{m+1} \le 4^{-(m+1)}, \\
& p_{m+1}^{-1}(\{0\} \times W \times \{2^{-(m+1)}\}) && \text{if } x_{m+1} > 4^{-(m+1)}, 
\end{aligned} 
\right. 
\\
V_1 &= \left\{ 
\begin{aligned}
& p_{m+1}^{-1}([0,m^{-1}] \times W \times \{0\}) &&\text{if } x_{m+1} \le 4^{-(m+1)}, \\
& p_{m+1}^{-1}([0,m^{-1}] \times W \times \{2^{-(m+1)}\}) && \text{if } x_{m+1} > 4^{-(m+1)}.
\end{aligned}
\right.
\end{align*}
Note that all the sets $V$, $V_1$ and $V_0$ are contained in $U(x,r) \cap X$. 
As written in \cite{S}, $V_1$ is a strong deformation retract of $V$ by a deformation $h : V \times [0,1] \to V$ sliding along the $(m+2)$-th coordiante, and $V_0$ is a strong deformation retract of $V_1$ by a deformation $g : V_1 \times [0,1] \to V_1$ sliding along the first coordinate. 
From the constructions, $h$ and $g$ are Lipschitz homotopies. 
Further, $V_0$ 
is Lipschitz contractible in itself. 
Hence, $V$ is Lipschitz contractible in $V$ to some point, by Lemma \ref{lem:DR}. 
Therefore, we conclude that $X$ is WLLC. 
By Lemma \ref{lem:WLLC to LLC}, $X$ is LLC. 
This completes the proof of Theorem \ref{thm:B}. 
\end{proof}




\subsection{Remark on the homology of normal currents}
By Lemma \ref{lem:I is cosheaf}, the functor $\N_\bullet^\cpt : \mathsf O(X) \to \cab$ is known to be a flabby cosheaf on each metric space $X$. 
Hence, we may apply Theorem \ref{thm:cosheaf} to this cosheaf $\N_\bullet^\cpt$. 
Now, we note that the space $\N_0^\cpt(X)$ is identified with the space $\mathcal M(X)$ of all finite signed Borel measures on $X$ with compact support. 
Therefore, when $X$ is LLC, we can guess that the homology $H_\ast(\N_\bullet^\cpt(X))$ coincides with the homology of some chain complex $\mathbf C_\bullet(X)$ such that its $0$-th group $\mathbf C_0(X)$ is $\mathcal M(X)$. 
Further, the homology $H_\ast(\mathbf C_\bullet(X))$ should be a topological invariant. 
Such a chain complex actually exists, called the measure chian complex, introduced by Thurston \cite{T}. 
However, there is no canonical map between $\mathbf C_\bullet(X)$ and $\N_\bullet^\cpt(X)$. 
We discuss such a difficult point in another paper. 

\subsection{Localizations}
\begin{lemma} \label{lem:snake}
Let $A_\bullet$ and $B_\bullet$ be chian complices of indexed by integers and $f : A_\bullet \to B_\bullet$ a chain map. 
Suppose that $H_m(f) : H_m(A_\bullet) \to H_m(B_\bullet)$ is surjective and $H_{m-1}(f) : H_{m-1}(A_\bullet) \to H_{m-1}(B_\bullet)$ is injecitve. 
Then, for any $b \in B_m$ and $a \in A_{m-1}$ with $\partial b = f(a)$ and $\partial a = 0$, there are $\bar a \in A_m$ and $\bar b \in B_{m+1}$ such that $\partial \bar b = b+f(\bar a)$. 
\end{lemma}
\begin{proof}
Let us take $b \in B_m$ and $a \in A_{m-1}$ with $\partial b = f(a)$ and $\partial a = 0$. 
Since $H_{m-1}(f)$ is injective, there is $\bar a \in A_m$ such that $\partial \bar a = a$. 
Then, we have $\partial (b- f(\bar a))=0$. 
Since $H_m(f)$ is surjective, there exist $\bar {\bar a} \in A_m$ and $\bar b \in B_{m+1}$ such that $\partial \bar {\bar a}=0$ and $f(\bar {\bar a}) + f(\bar a) = b + \partial \bar b$.
This completes the proof.
\end{proof}

As a corollary to the proof of Theorem \ref{thm:cosheaf}, we have 
\begin{corollary} \label{cor:localization}
Let $m \ge 1$ and $\epsilon > 0$. 
Let $X$ be a metric space which is $\tilde H_j^\Lip$-locally trivial for $0 \le j \le m-1$. 
\begin{itemize}
\item[(1)] If $X$ is $H_k^\sing$-locally trivial for $1 \le k \le m$, then for $c \in S_m(X)$ with $\partial c \in S_{m-1}^\Lip(X)$, there exist finitely many elements $c_1, \dots, c_N \in S_m(X)$, $c_1^L, \dots, c_N^L \in S_m^\Lip(X)$ and $\bar c_1, \dots, \bar c_N \in S_{m+1}(X)$ such that 
\begin{itemize}
\item[(1-a)] there is $n \ge 0$ such that $\sum_i c_i = \mathrm{sd}^n (c)$ and $\partial \bar c_i = c_i - c_i^L$; 
\item[(2-a)] $\mathrm{im}(\bar c_i) \subset U(\mathrm{im}(c), \epsilon)$ and $\mathrm{diam}\, \mathrm{im}(\bar c_i) < \epsilon$.
\end{itemize}
\item[(2)] If $X$ is $H_k^\IC$-locally trivial for $1 \le k \le m$, then for $T \in \I_m^\cpt(X)$ and for $c \in S_{m-1}^\Lip(X)$ satisfying $\partial T = [c]$ and $\partial c = 0$, there exist finitely many elements $T_1, \dots, T_N \in \I_m^\cpt(X)$, $c_1^L, \dots, c_N^L \in S_m^\Lip(X)$ and $S_1, \dots, S_N \in \I_{m+1}^\cpt(X)$ such that 
\begin{itemize}
\item[(2-a)] $\sum_i T_i = T$ and $\partial S_i = T_i - [c_i^L]$; 
\item[(2-b)] there is $n \ge 0$ such that $\sum_i c_i^L = \mathrm{sd}^n c$; 
\item[(2-c)] $\mathrm{spt}(S_i) \cup \mathrm{im}(c_i^L) \subset U(\mathrm{spt}(T) \cup \mathrm{im}(c), \epsilon)$ and $\mathrm{diam}\, \mathrm{spt}(S_i) < \epsilon$.
\end{itemize}
\end{itemize}
\end{corollary}
\begin{proof}
Let $X$ be as in the assumption. 
Let us prove (1). 
We suppose that $X$ is $H_k^\sing$-locally trivial for $1 \le k \le m$ and take $c \in S_m(X)$ with $\partial c \in S_{m-1}^\Lip(X)$. 
Let us take a finite open covering $\mathcal U = \{U_i\}_{i=1}^N$ of $\mathrm{im}(c)$ such that $\mathrm{diam}(U_i) < \epsilon/2$.
We set $U = \bigcup_{i=1}^N U_i$. 
Then, we may regard $c$ as an element in $c \in S_m(U)$. 
Further, the same symbol $c$ denotes the element in $\mathfrak S_m(U)$ represented by $c$. 
Since $\partial c \in \mathfrak S_{m-1}^\Lip(U)$, by Lemma \ref{lem:double}, 
there are elements $c_{k, m-k-1}^L \in \check C_k(\mathcal U, \mathfrak S_{m-k-1}^\Lip)$ such that 
\[
\varepsilon c_{0,m-1}^L = \partial c \text{ and } 
\Phi(c_{k,m-1}^L) = \partial c_{k-1,m-k}^L
\]
for $1 \le k \le m-1$. 
On the other hands, since $\varepsilon : \check C_0(\U, \mathfrak S_m) \to \mathfrak S_m(U)$ is surjective, there is $c_{0,m} \in \check C_0(\U,\mathfrak S_m)$ such that $\varepsilon c_{0,m} = c$. 
By the choice, we have 
\[
\varepsilon (\partial c_{0,m}-c_{0,m-1}^L) = 0.
\] 
Therefore, we obtain $c_{1,m-1} \in \check C_1(\U,\mathfrak S_{m-1})$ satisfying 
\[
\Phi c_{1,m-1} = \partial c_{0,m} - c_{0,m-1}^L. 
\]
By repeating such an argument, we have a sequence of elements $c_{k,m-k} \in \check C_k(\U,\mathfrak S_{m-k})$ such that 
\[
\Phi c_{k,m-k} = \partial c_{k-1,m-k+1} + (-1)^k c_{k-1, m-k}^L
\]
fro all $1 \le k \le m$.
Then, since $\mathfrak S_0 = \mathfrak S_0^\Lip$, the element $c_{m,0}$ belongs to $\check C_m(\U,\mathfrak S_0^\Lip)$.
Further, we have $\partial c_{m-1,1} \in \check C_{m-1}(\U,\mathfrak S_0^\Lip)$.
Since $H_1^\sing \cong H_1^\Lip$, by Lemma \ref{lem:snake}, there are elements $c_{m-1,1}^L \in \check C_{m-1}(\U,\mathfrak S_1^\Lip)$ and $c_{m-1.2} \in \check C_{m-1}(\U,\mathfrak S_2)$ such that 
\[
\partial c_{m-1,2} = c_{m-1,1} - c_{m-1,1}^L.
\]
Hence, we obtain 
\[
\partial (\Phi c_{m-1,2} - c_{m-2,2}) = (-1)^{m-1} c_{m-2,1}^L - \Phi c_{m-1,1}^L 
\]
which is an element of $\check C_{m-2}(\U, \mathfrak S_1^\Lip)$.
By Lemma \ref{lem:snake}, there are elements 
$c_{m-2,3} \in \check C_{m-2}(\U,\mathfrak S_3)$ and 
$c_{m-2,2}^L \in \check C_{m-2}(\U,\mathfrak S_2)$ such that 
\[
\partial c_{m-2,3} = \Phi c_{m-1,2} - c_{m-2,2} - c_{m-2,2}^L.
\]
By using Lemma \ref{lem:snake} repeatedly, 
we have sequences of elements $c_{m-k,k+1} \in \check C_{m-k}(\U,\mathfrak S_{k+1})$ and $c_{m-k,k}^L \in \check C_{m-k}(\U,\mathfrak S_k^\Lip)$ such that 
\begin{align*}
&\partial c_{m-k,k+1} = \Phi c_{m-k+1,k} + (-1)^{k+1} c_{m-k,k} - c_{m-k,k}^L 
\end{align*}
for $2 \le k \le m$.
Let us set 
\[
\tilde c_{0,m} :=  (-1)^{m+1} \Phi c_{1,m} + c_{0,m} \in \check C_0(\U, \mathfrak S_m).
\]
Then, we have
\begin{equation} \label{eq:conc}
\left\{
\begin{aligned}
& \varepsilon \tilde c_{0,m} = c, \\
& \partial c_{0,m+1} = (-1)^{m+1} \tilde c_{0,m} - c_{0,m}^L.
\end{aligned}
\right.
\end{equation} 
Further, let us set $\tilde c_{0,m} = (c_i)_{i=1}^N$, $(-1)^{m+1} c_{0,m+1} = (\bar c_i)_{i=1}^N$ and $(-1)^{m+1} c_{0,m}^L = (c_i^L)_{i=1}^N$, where $c_i \in \mathfrak S_m(U_i)$, $\bar c_i \in \mathfrak S_{m+1}(U_i)$ and $c_i^L \in \mathfrak S_m^\Lip(U_i)$. 
Then, the relation \eqref{eq:conc} is translated as 
\[
\begin{aligned}
& \sum_{i=1}^N c_i = c \in \mathfrak S_m(U), \\
& \partial \bar c_i = c_i - c_i^L \in \mathfrak S_m(U_i)
\end{aligned}
\]
for every $i \in \{1, \dots, N\}$. 
We consider representatives of $c_i$ in $S_m(U_i)$, $c_i^L$ in $S_m^\Lip(U_i)$ and $\bar c_i$ in $S_{m+1}(U_i)$. 
Taking subdivision of them sufficiently many times, we obtain the conclusions of the statement (1).

The statement (2) can be proved by an argument similar to the proof of (1).
This completes the proof of Theorem \ref{cor:localization}
\end{proof}

Remark that Corollary \ref{cor:localization} is a generalization of statements in \cite{RS} in terms of local triviality of homology theories. 

\vspace{10pt}
\noindent{\bf Acknowledgements}.
The author would like to express my appreciation to Takumi Yokota for helpful comments, discussions, and careful reading of a preliminary version of the manuscript.
He is also grateful to Shouhei Honda, Yu Kitabeppu, Takashi Shioya, and Takao Yamaguchi for helpful comments and for discussions.

\end{document}